\documentclass[a4paper, 12pt, oneside]{amsart}

\usepackage[T1]{fontenc}
\usepackage[utf8]{inputenc}
\usepackage{mathptmx}       
\usepackage{helvet}         
\usepackage{courier}        
\usepackage{type1cm}        
\usepackage{graphicx}        

\usepackage{amssymb, amsmath}
\usepackage{enumitem}
\usepackage[DIV=10]{typearea}
\usepackage{mathtools}

\newcommand{\Z}{\mathbb{Z}}
\newcommand{\N}{\mathbb{N}}

\newcommand{\G}{\mathcal{G}}
\newcommand{\C}{\mathbb{C}}
\newcommand{\K}{\mathcal{K}}

\newcommand{\id}{{\operatorname{Id}}}

\newcommand{\TSS}{{\overline{X}}}
\newcommand{\tss}{{\overline{\sigma}}}

\newcommand{\supp}{\operatorname{supp}}
\newcommand{\cls}[1][(R,\Gamma)]{\mathcal{C}_{#1}}

\theoremstyle{theorem}
\newtheorem{theorem}{Theorem}[section]
\newtheorem{proposition}[theorem]{Proposition}
\newtheorem{corollary}[theorem]{Corollary}
\newtheorem{lemma}[theorem]{Lemma}

\theoremstyle{definition}
\newtheorem{definition}[theorem]{Definition}
\newtheorem{remark}[theorem]{Remark}

\title{Diagonal-preserving graded isomorphisms of Steinberg algebras}

\author{Toke Meier Carlsen}
\address{Department of Science and Technology\\University of the Faroe Islands\\
N\'oat\'un 3\\ FO-100 T\'orshavn\\ Faroe Islands}
\email{toke.carlsen@gmail.com}

\author{James Rout}
\address{School of Mathematics and Applied Statistics\\ University of Wollongong\\
NSW 2522\\ Australia}
\email{jdr749@uowmail.edu.au}

\date{\today}
\thanks{The second author was partially supported by the ARC grant DP150101595 to Aidan Sims, the Simons-Fundation grant 346300 and the Polish Government MNiSW 2015-2019 matching fund.}

\begin{document}
	
\maketitle

\begin{abstract}
	We study Steinberg algebras constructed from ample Hausdorff groupoids over commutative integral domains with identity. We reconstruct (graded) groupoids from (graded) Steinberg algebras and use this to characterise when there is a diagonal-preserving (graded) isomorphism between two (graded) Steinberg algebras. We apply this characterisation to groupoids of directed graphs in order to study diagonal-preserving (graded) isomorphisms of Leavitt path algebras and $*$-isomorphisms of graph $C^*$-algebras.
\end{abstract}

\section{Introduction}


Steinberg algebras, introduced in \cite{S} and independently in \cite{CFST}, are algebraic analogues of groupoid $C^*$-algebras. The class of Steinberg algebras includes for instance discrete inverse semigroup algebras (see for example \cite{S}), Kumjian--Pask algebras (see for example \cite{CP}), and Leavitt path algebras (see for example \cite{CS}). Steinberg algebras have recently attracted a great deal of attention (see for instance \cite{ABHS,CE,CEHS,CEP,CMMS,S1,S2}).

For an ample Hausdorff groupoid $\G$ and a ring $R$ with identity, the Steinberg algebra $A_R(\G)$ is the convolution algebra of locally constant functions from $\G$ to $R$. When $R=\C$, the Steinberg algebra $A_\C(\G)$ is a dense subalgebra of the groupoid $C^*$-algebra $C^*(\G)$. The algebra of locally constant functions from the unit space $\G^0$ of $\G$ to $R$ is a commutative subalgebra of $A_R(\G)$, called the diagonal and denoted $D_R(\G)$. If $\Gamma$ is a discrete group, then any continuous cocycle (i.e., a groupoid homomorphism) $c:\G\to\Gamma$ induces a $\Gamma$-grading of $A_R(\G)$.

We show that for a commutative integral domain with identity $R$, an ample Hausdorff groupoid $\G$, a discrete group $\Gamma$, and a continuous cocycle $c:\G\to\Gamma$ satisfying a mild condition, the groupoid $\G$ and the cocycle $c$ can be recovered from the Steinberg algebra $A_R(\G)$, its diagonal subalgebra $D_R(\G)$, and the $\Gamma$-grading of $A_R(\G)$. It follows that if $\G_1$ and $\G_2$ are two ample Hausdorff groupoids and $c_1:\G_1\to\Gamma$ and $c_2:\G_2\to\Gamma$ are continuous cocycles satisfying the aforementioned condition, then there is a $\Gamma$-graded isomorphism $\phi:A_R(\G_1)\to A_R(\G_2)$ satisfying $\phi(D_R(\G_1))=D_R(\G_2)$ (we call such an isomorphism \emph{diagonal-preserving}), if and only if there is a topological groupoid isomorphism $\psi:\G_1\to\G_2$ such that $c_2\circ\psi=c_1$. 

It is worth noting that we only need the ring structure of $A_R(\G)$ and $D_R(\G)$ to recover $\G$. It follows that we only need a diagonal-preserving $\Gamma$-graded ring-isomorphism $\phi:A_R(\G_1)\to A_R(\G_2)$ in order to conclude that $\G_1$ and $\G_2$ are isomorphic as graded topological groupoids. On the other hand, if $\G_1$ and $\G_2$ are isomorphic as graded topological groupoids, then it follows that $A_R(\G_1)$ and $A_R(\G_2)$ are isomorphic as $*$-algebras by a diagonal-preserving isomorphism. Thus, if there is a diagonal-preserving $\Gamma$-graded ring-isomorphism between $A_R(\G_1)$ and $A_R(\G_2)$, then there is also a diagonal-preserving $\Gamma$-graded $*$-algebra-isomorphism between $A_R(\G_1)$ and $A_R(\G_2)$.

If we let $\Gamma$ be the trivial group, then the cocycles and the gradings also become trivial, and our result then says that two groupoids $\G_1$ and $\G_2$ belonging to a large class of ample Hausdorff groupoids, which contains all topologically principal ample Hausdorff groupoids and any ample Hausdorff groupoid all of whose isotropy subgroups are either free groups or free abelian groups, are isomorphic if and only if there is a diagonal-preserving isomorphism between $A_R(\G_1)$ and $A_R(\G_2)$. 

\subsection{Motivation and historical background}
When $\G$ is the graph groupoid $\G_E$ of a directed graph $E$, the Steinberg algebra $A_R(\G_E)$ is precisely the Leavitt path algebra $L_R(E)$, and the diagonal subalgebra $D_R(\G_E)$ is precisely the diagonal subalgebra $D_R(E)$ of $L_R(E)$ (\cite[Remark~4.4]{CFST} and \cite[Example~3.2]{CS}).



Working with the Steinberg algebra model of a Leavitt path algebra, Brown, Clark and an Huef showed in \cite{BCH} that if $E$ is a row-finite directed graph without sinks (sources using their convention) and $R$ is a commutative integral domain with identity, then the graph groupoid $\G_E$ can be reconstructed from the data $(L_R(E),D_R(E))$ (\cite[Theorem~4.9]{BCH}). They used this reconstruction result to show that, for row-finite directed graphs without sinks, there is a diagonal-preserving $*$-isomorphism between Leavitt path algebras if and only if there is an isomorphism between the corresponding graph groupoiods (\cite[Theorem~6.2]{BCH}), and deduced that this is equivalent to there being a diagonal-preserving isomorphism between the corresponding graph $C^*$-algebras (\cite[Corollary~6.3]{BCH}). 

Ara, Bosa, Hazrat and Sims showed in \cite{ABHS} that if $\G$ is an ample Hausdorff groupoid equipped with a cocycle whose kernel is topologically principal, and $R$ is a commutative integral domain with identity, then $\G$ can be reconstructed from the data $(A_R(\G),D_R(\G))$ (\cite[Corollary~3.11]{ABHS}). They used this reconstruction result to show that there is a diagonal-preserving graded isomorphism of Steinberg algebras if and only if there is a cocycle-preserving isomorphism between the corresponding groupoids (\cite[Theorem~3.1]{ABHS}). By applying this to graph groupoids, they deduced, among many other interesting results, that, for directed graphs in which every cycle has an exit, there is a diagonal-preserving isomorphism between Leavitt path algebras if and only if there is a diagonal-preserving $*$-isomorphism between the corresponding graph $C^*$-algebras (\cite[Corollary~4.4]{ABHS}). 


\subsection{Presentation of the results}
In this paper, we use techniques and ideas of \cite{ABHS} and \cite{BCH} to extend \cite[Theorem~3.1]{ABHS} to a larger class of ample Hausdorff groupoids by significantly relaxing the assumption that the kernel of the cocycle is topologically principal (Theorem~\ref{thm:steinberg}). We also prove a ``stabilised version" of this result (Theorem~\ref{thm:stable}) and by combining this with \cite[Theorem 3.2]{CRS} we are able to relate groupoid equivalence and Kakutani equivalence with diagonal-preserving isomorphism of stabilised Steinberg algebras (Corollary~\ref{cor:stable}). We also discuss how diagonal-preserving graded isomorphisms of Steinberg algebras are related to actions of inverse semigroups (Corollary~\ref{cor:action}). 

All of these results hold for arbitrary ample Hausdorff groupoids $\G$ and cocycles $c:\G\to\Gamma$ with values in a discrete group, satisfying a mild condition, which in particular is satisfied if there is a dense subset $X \subseteq \G^0$ of the unit space of $\G$ such that $c^{-1}(e)\cap\G_x^x$ (where $\G_x^x$ is the isotropy group) is a free or free abelian group for all $x \in X$ (for Corollary~\ref{cor:stable} we also require that $\G^0$ be $\sigma$-compact). In particular, these results hold for the groupoid of an arbitrary directed graph, the groupoid of a finitely-aligned higher-rank graph, the groupoid of a partial action of a free or a free abelian group on a totally disconnected locally compact Hausdorff space, and many tight groupoids of inverse semigroups (see Remark~\ref{cls}).

In the second half of the paper, we apply our results to the groupoids of directed graphs in order to strengthen \cite[Corollary 4.4]{ABHS}, \cite[Theorem 5.3]{AER}, \cite[Theorem 6.2]{BCH}, \cite[Corollary 6.1]{CEOR}, and \cite[Corollary 6.4]{CEOR}, and to obtain Leavitt path algebra versions of the results of \cite{CR}. In particular, we show that if $R$ is a commutative integral domain with identity and $E$ and $F$ are arbitrary directed graph, then there is a diagonal-preserving isomorphism between $L_R(E)$ and $L_R(F)$ if and only if the graph groupoids $\G_E$ and $\G_F$ are isomorphic (Corollary~\ref{cor:orbit}). By \cite{BCW}, the latter condition is equivalent to the existence of a diagonal-preserving isomorphism between the graph $C^*$-algebras $C^*(E)$ and $C^*(F)$, and by \cite{AER}, it is also equivalent to the existence of an orbit equivalence from $E$ to $F$ that preserves isolated eventually periodic points. We also show that there is a diagonal-preserving isomorphism between the stabilised Leavitt path algebras $L_R(E)\otimes M_\infty (R)$ and $L_R(F)\otimes M_\infty (R)$ if and only if the groupoids $\G_E$ and $\G_F$ are groupoid equivalent (Corollary~\ref{thm:4}). By \cite{CRS}, the latter condition is equivalent to $\G_E$ and $\G_F$ being Kakutani equivalent, and to the existence of a diagonal-preserving isomorphism between the stabilised graph $C^*$-algebras $C^*(E)\otimes\K$ and $C^*(F)\otimes\K$. It is shown in \cite{CEOR} that if $E$ and $F$ are finite and have no sinks or sources, then the latter condition is equivalent to the two-sided edge shifts of $E$ and $F$ being flow equivalent. In this case, we also prove, by using \cite{CR}, that if $L_R(E)$ and $L_R(F)$ are equipped with the standard $\mathbb{Z}$-grading, then there is a diagonal-preserving graded isomorphism between the stabilised Leavitt path algebras $L_R(E)\otimes M_\infty (R)$ and $L_R(F)\otimes M_\infty (R)$ if and only if the two-sided edge shifts of $E$ and $F$ are conjugate (Corollary~\ref{thm:43}).



\section{Ample Hausdorff groupoids and their Steinberg algebras}

For the benefit of the reader, we recall in this section the definitions of ample Hausdorff groupoids and Steinberg algebras. This is standard and can be found in many other papers, for example \cite{ABHS} and \cite{BCH}.


\subsection{Ample Hausdorff groupoids}

A \emph{groupoid} is a small category $\G$ with inverses. The \emph{unit space} $\G^0$ of $\G$ is the collection of identity morphisms. That is, $\G^0 = \{ \eta \eta^{-1}: \eta \in \G \}$. The \emph{range map} $r: \G \to \G^0$ is given by $r(\eta) := \eta \eta^{-1}$ and the \emph{source map} $s: \G \to \G^0$ is given by $s(\eta) = \eta^{-1} \eta$. We say a pair $(\alpha,\beta) \in \G \times \G$ is \emph{composable} is $s(\alpha) = r(\beta)$, and denote by $\G^2$ the collection of all composable pairs.

For $U,V \subseteq \G$, we write $UV := \{ \alpha \beta: \alpha \in U, \beta \in V, \text{ and } s(\alpha) = r(\beta) \}$. For $U \subseteq \G$, we write $U^{-1} := \{ \eta^{-1}: \eta \in U \}$. Given units $x,y \in \G^0$, we write $\G_x := s^{-1}(x)$, $\G^y= r^{-1}(y)$, and $\G^x_y  := \G_x \bigcap \G^y$. Note that $\G_x^x$ is a group called the \emph{isotropy group} for $x \in \G^0$.

We say that $\G$ is a Hausdorff groupoid if it has a Hausdorff topology under which the range, source and inverse maps are continuous, and the composition map is continuous with respect to the subspace topology on $\G^2 \subseteq \G \times \G$. We say that $\G$ is \emph{\'{e}tale} if the range and source maps are each local homeomorphisms. If $\G$ is \'{e}tale, then $\G^0$ is clopen in $\G$. We say that $\G$ is \emph{ample} if it is \'{e}tale and $\G^0$ has a basis consisting of compact open sets. A \emph{bisection} is a subset $U \subseteq \G$ such that the range and source maps restrict to homeomoprhisms from $U$ onto open subsets of $\G^0$. If $\G$ is ample, then it admits a basis consisting of compact open bisections.

Let $\Gamma$ be a discrete group. A \emph{continuous cocycle} $c$ is a continuous homomorphism from $\G$ to $\Gamma$ (that is, $c$ carries composition in $\G$ to the group operation in $\Gamma$).

An isomorphism of Hausdorff groupoids $\phi:\G \to \mathcal{H}$ is a homeomorphism from $\G$ to $\mathcal{H}$ that carries units to units, preserves the range and source maps, and satisfies $\phi(\alpha \beta) = \phi(\alpha)\phi(\beta)$ for all composable pairs $(\alpha,\beta) \in \G^2$. The uniqueness of inverses implies that $\phi(\alpha^{-1})=\phi(\alpha)^{-1}$ for $\alpha \in \G$.

\subsection{Steinberg algebras}

Let $R$ be a commutative ring with identity and $\G$ an ample Hausdorff groupoid. The \emph{Steinberg algebra} $A_R(\G)$ of $\G$ is the collection of locally constant $R$-valued functions on $\G$ with compact support with pointwise addition, and convolution product $(f * g) (\eta)=\sum_{\alpha \beta = \eta} f(\alpha)g(\beta)$. For $f \in A_R(\G)$, we write $\supp(f) = \{ \eta \in \G: f(\eta) \ne 0 \}$.

Denote by $S_\G$ the collection of all compact open bisections of $\G$. Note that $A_R(\G) = \operatorname{span}_R \{1_U: U \in S_\G \}$ by \cite[Proposition~4.3]{S}, and $1_U 1_V = 1_{UV}$ for $U,V \in S_\G$ by \cite[Proposition~3.5(3)]{S}.

Let $\Gamma$ be a discrete group and $c: \G \to \Gamma$ a continuous cocycle. By \cite[Lemma~3.1]{CS} there is a $\Gamma$-grading of $A_R(\G)$ such that $A_R(\G)_g = \{ f \in A_R(\G): \supp(f) \subseteq c^{-1}(g) \}$ for $g \in \Gamma$.

We write $D_R(\G) := A_R(\G^0)$ for the commutative algebra of locally constant compactly supported functions from $\G^0$ to $R$ under pointwise operations. Since $\G^0$ is clopen there is an embedding $\iota: A_R(\G^0) \to A_R(\G)$ such that $\iota(f)|_{\G^0} = f$ and $\iota(f)|_{\G \backslash \G^0} = 0$. With this embedding we regard $D_R(\G)$ as a commutative subalgebra of $A_R(\G)$ and call it the \emph{diagonal} of $A_R(\G)$. When $R$ is a commutative integral domain, we have that $D_R(\G) =\operatorname{span}_R\{1_U: U \subseteq \G^0 \text{ is compact open} \}.$

For ample Hausdorff groupoids $\G_1$ and $\G_2$ and continuous cocycles $c_1:\G_1\to\Gamma$, $c_2:\G_2\to\Gamma$, we say that an isomorphism $\phi: A_R(\G_1) \to A_R(\G_2)$ is \emph{diagonal-preserving} if $\phi(D_R(\G_1))=D_R(\G_2)$, and \emph{graded} if $\phi(A_R(\G_1)_g) = A_R(\G_2)_g$ for $g \in \Gamma$.

If $R$ is a $*$-ring (i.e., $R$ comes with an involution $r\mapsto r^*$ that is also a ring automorphism), then $A_R(\G)$ is a $*$-algebra with the involution $f\mapsto f^*$ defined by $f^*(\eta)=(f(\eta^{-1}))^*$ for $f\in A_R(\G)$ and $\eta\in \G$. We assume throughout the paper that every ring $R$ is a $*$-ring (the involution could be the identity map).

\section{Recovering $\G$ from $(A_R(\G),D_R(\G))$}

In this section we strengthen \cite[Theorem 3.1]{ABHS} by significantly relaxing the assumption that $c^{-1}(e)$ and $d^{-1}(e)$ are topologically principal (Theorem~\ref{thm:steinberg}). We also prove a ``stabilised version" of Theorem~\ref{thm:steinberg} (Theorem~\ref{thm:stable}). By combining this with \cite[Theorem 3.2]{CRS}, we relate groupoid equivalence and Kakutani equivalence with diagonal-preserving isomorphisms of stabilised Steinberg algebras (Corollary~\ref{cor:stable}). We end the section by relating diagonal-preserving graded isomorphisms of Steinberg algebras with actions of inverse semigroups (Corollary~\ref{cor:action}).

Let $R$ be a commutative integral domain with identity $1$ and $G$ a group with identity $e$. The \emph{group-ring} $RG$ of $G$ is defined by $RG := \{ \sum_{i=1}^n r_i g_i: r_i \in R, g_i\in G\}$. An element $a \in RG$ is a \emph{unit} if there exist $b,c \in RG$ such that $ab = 1 = ca$. A unit is \emph{trivial} if it has the form $ug$ for some $u \in R$ and $g\in G$. Througout this section we will work with groups $G$ that have the property that the group-ring $RG$ has no zero-divisors and only trivial units. A group $G$ is \emph{indexed} if there is a homomorphism $\gamma: G \to \mathbb{Q}^+$ such that $\gamma(G) \not = \{ 0\}$. A group $G$ is \emph{indicable throughout} if every subgroup $H \subseteq G$ such that $H \not = \{e\}$ is indexed. For example, free groups and free Abelian groups are indicable throughout. If $G$ is indicable througout, then $RG$ has no zero-divisors and only trivial units (\cite[Theorem~12 and Theorem~13]{Higman}). Notice that if $H$ is a subgroup of $G$, and $RG$ has no zero-divisors and only trivial units, then $RH$ has no zero-divisors and only trivial units.

In this section we prove the following theorem.

\begin{theorem}\label{thm:steinberg}
Let $R$ be a commutative integral domain with identity and let $\Gamma$ be a discrete group. For $i=1,2$, let $\G_i$ be an ample Hausdorff groupoid and $c_i:\G_i\to\Gamma$ a continuous cocycle such that there is a dense subset $X_i \subseteq \G_i^0$ such that the group-ring $R(c_i^{-1}(e)\cap (\G_i)_x^x)$ has no zero-divisors and only trivial units for all $x \in X_i$. The following are equivalent.
\begin{enumerate}
	\item There is an isomorphism $\psi:\G_1\to\G_2$ such that $c_2\circ\psi=c_1$.
	\item There is a diagonal-preserving graded $*$-algebra-isomorphism between $A_R(\G_1)$ and $A_R(\G_2)$.
	\item There is a diagonal-preserving graded ring-isomorphism between $A_R(\G_1)$ and $A_R(\G_2)$.
\end{enumerate}
\end{theorem}

It is perhaps worth mentioning that if $\Gamma$ is the trivial group, then the cocycles $c_i$ and the gradings of $A_R(\G_i)$ are trivial, and we obtain the following corollary from Theorem~\ref{thm:steinberg}.

\begin{corollary}\label{cor:steinberg}
Let $R$ be a commutative integral domain with identity. For $i=1,2$, let $\G_i$ be an ample Hausdorff groupoid such that there is a dense subset $X_i \subseteq \G_i^0$ such that the group-ring $R((\G_i)_x^x)$ has no zero-divisors and only trivial units for all $x \in X_i$. The following are equivalent.
\begin{enumerate}
	\item $\G_1$ and $\G_2$ are isomorphic.
	\item There is a diagonal-preserving $*$-algebra-isomorphism between $A_R(\G_1)$ and \linebreak$A_R(\G_2)$.
	\item There is a diagonal-preserving ring-isomorphism between $A_R(\G_1)$ and $A_R(\G_2)$.
\end{enumerate}
\end{corollary}

For a commutative integral domain $R$ with identity and a discrete group $\Gamma$, let $\cls$ be the class of pairs $(\G,c)$, where $\G$ is an ample Hausdorff groupoid and $c$ is a continuous cocycle from $\G$ to $\Gamma$ such that there is a dense subset $X \subseteq \G^0$ such that the group-ring $R(c^{-1}(e)\cap(\G)_x^x)$ has no zero-divisors and only trivial units for all $x \in X$. Theorem~\ref{thm:steinberg} then says that for $(\G_1,c_1),(\G_2,c_2)\in\cls$, there is a diagonal-preserving graded isomorphism between $A_R(\G_1)$ and $A_R(\G_2)$ if and only if there is an isomorphism $\psi:\G_1\to\G_2$ such that $c_2\circ\psi=c_1$.

\begin{remark}\label{cls}
Let $R$ be any commutative integral domain with identity and $\Gamma$ any discrete group. The following pairs belong to $\cls$.
\begin{itemize}
	\item The pair $(\G_E,c)$, where $\G_E$ is the groupoid of a directed graph $E$ (see for example \cite{AER,BCW,Car,CR,CW,CS}) and $c$ is any continuous cocycle from $\G_E$ to $\Gamma$ (see the proof of Theorem~\ref{thm:1});
	\item the pair $(\G_\Lambda,c)$, where $\G_\Lambda$ is the groupoid of a finitely-aligned higher-rank graph $\Lambda$ (see for example \cite{CP,Y}) and $c$ is any continuous cocycle from $\G_\Lambda$ to $\Gamma$;
	\item the pair $(\G_{\text{tight}}(T),c)$, where $\G_{\text{tight}}(T)$ is the tight groupoid constructed from a Boolean dynamical system in \cite{COP}, and $c$ is any continuous cocycle from $\G_{\text{tight}}(T)$ to $\Gamma$;
	\item the pair $(\G_\theta,c_\phi)$, where $\G_\theta$ is the groupoid constructed in \cite{Abadie} from a partial action $\theta$ of a discrete group $G$ on a totally disconnect locally compact Hausdorff space $X$, and $c_\phi:\G_\theta\to\Gamma$ is the cocycle induced by a group homomorphism $\phi:G\to\Gamma$ such that the group-ring $R(\ker\phi)$ has no zero-divisors and only trivial units;
	\item the pair $(\G_{\text{tight}}(\mathcal{S}),c_\phi)$, where $\G_{\text{tight}}(\mathcal{S})$ is the tight groupoid of an $E^*$-unitary inverse semigroup (see \cite{Exel,EP}), and $c_\phi$ is the cocycle induced by a homomorphism $\phi:\mathcal{S}\to\Gamma$ such that there is a dense subset $X \subseteq \hat{\mathcal{E}}_{\text{tight}}$ such that the group-ring $R(G_x)$, where $G_x$ is the group $\{[s,x]\in\G_{\text{tight}}(\mathcal{S}): \phi(s)=e,\ \theta_s(x)=x\}$, has no zero-divisors and only trivial units for all $x \in X$.
\end{itemize}
\end{remark}

The implication $(1)\implies (2)$ in Theorem~\ref{thm:steinberg} follows easily from the definitions of $A_R(\G)$, $D_R(\G)$ and the grading on $A_R(\G)$, and the implication $(2)\implies (3)$ is obvious. Our strategy for proving $(3)\implies (1)$ is similar to the strategy used to prove \cite[Theorem 6.2]{BCH}. For $(\G,c)\in\cls$, we construct from $A_R(\G)$ and $D_R(\G)$ a graded ample Hausdorff groupoid $W_\G$ such that $W_\G$ is isomorphic to $\G$ (Proposition~\ref{pro:1}), and then show that a diagonal-preserving graded ring-isomorphism between $A_R(\G_1)$ and $A_R(\G_2)$ induces a graded isomorphism between $W_{\G_1}$ and $W_{\G_2}$ (Proposition~\ref{pro:3}). Since we are working with isomorphisms that are not necessarily $*$-isomorphisms, and since we also need to recover the grading of $\G$ given by the cocycle $c$, we use a definition of normalisers similar to the one used in \cite{ABHS} instead of the definition used in \cite{BCH}.

\subsection{The normalisers of $D_R(\G)$}

In the following, $R$ is a commutative integral domain with identity, $\Gamma$ is a discrete group, and $(\G,c)\in\cls$.

\begin{definition}
A \emph{normaliser} of $D_R(\G)$ in $A_R(\G)$ is a pair $(m,n) \in A_R(\G) \times A_R(\G)$ satisfying $m D_R(\G) n \bigcup n D_R(\G) m \subseteq D_R(\G)$ and $mnm=m$ and $nmn=n$. We denote the set of normalisers $N_R(\G)$ and let $N_R(\G)_g:=N_R(\G)\cap (A_R(\G)_g\times A_R(\G)_{g^{-1}})$ for $g\in\Gamma$.
\end{definition}

Since $D_R(\G)=\operatorname{span}_R\{1_U:U \subseteq \G^0 \text{ is compact open} \}$, it is straightforward to check that if $A$ is a compact open bisection of $\G$, then $(1_A,1_{A^{-1}})\in N_R(\G)$. If in addition $c(A)=\{g\}$, then $(1_A,1_{A^{-1}})\in N_R(\G)_g$.

\begin{lemma}\label{lem:1}
	Let $g\in\Gamma$ and suppose $(m,n)\in N_R(\G)_g$. Then
	\begin{enumerate}
		\item $mn=1_{s(\supp(n))}$ and $nm=1_{s(\supp(m))}$.
		\item $\supp(m)$ is a compact open bisection contained in $c^{-1}(g)$.
		\item $\supp(n)=\supp(m)^{-1}$.
	\end{enumerate}
\end{lemma}

\begin{proof}
	(1): Since $m,n\in A_R(\G)$, it follows that both $\supp(m)$ and $\supp(n)$ are compact open. Thus, $1_{s(\supp(m))\cup r(\supp(n))}\in D_R(\G)$. It follows that $mn=m1_{s(\supp(m))\cup r(\supp(n))}n\in D_R(\G)$. Since $(mn)^2=mn$ and $R$ is an integral domain, it follows that $mn=1_U$ for some compact open subset $U$ of $\G^0$. Since $mn(\eta)=0$ if $s(\eta)\notin s(\supp(n))$ or $r(\eta)\notin r(\supp(m))$, it follows that $U\subseteq s(\supp(n))\cap r(\supp(m))$. Conversely, since $n=nmn=n1_U$ and $m=mnm=1_Um$, it follows that $s(\supp(n))\cup r(\supp(m))\subseteq U$. Thus, $mn=1_{s(\supp(n))}=1_{r(\supp(m))}$. That $nm=1_{s(\supp(m))}=1_{r(\supp(n))}$ can be proved similarly.
	
	(2): Using arguments similar to those used in the previous paragraph, one can show that if $U$ is a compact open subset of $s(\supp(m))=r(\supp(n))$, then $m1_Un=1_{U'}$ where $U'=s(\supp(n)\cap r^{-1}(U))=r(\supp(m)\cap s^{-1}(U))$.
	
	Since $\supp(n)$ is a union of a finite number of compact open bisections, it follows that there is a finite set $F$ of mutually disjoint compact open subsets of $r(\supp(n))$ such that $\bigcup_{U\in F}U=r(\supp(n))$, and such that $r(s^{-1}(\{x\})\cap\supp(n))\cap U$ consist of at most one point for each $x\in s(\supp(n))$ and each $U\in F$.
	
	Let $x\in r(\supp(m))=s(\supp(n))$ and suppose the group-ring $R(c^{-1}(e)\cap\G_{x}^{x})$ has no zero-divisors and only trivial units. Since $\sum_{U\in F}m1_Un=m1_{s(\supp(m))}n=mn=1_{s(\supp(n))}$, it follows that there is a unique $U_x\in F$ such that $x\in s(\supp(n)\cap r^{-1}(U_x))$. Let $y$ be the unique point in $r(s^{-1}(\{x\})\cap\supp(n))\cap U_x$. Let $E_n:=\{\eta\in\supp(n):s(\eta)=x,\ r(\eta)=y\}$, let $E_m:=\{\zeta\in\supp(m):s(\zeta)=y,\ r(\zeta)=x\}$ and choose an $\eta_0\in E_n$. Then $\eta_0^{-1}\eta, \zeta\eta_0\in c^{-1}(e)\cap\G_{x}^{x}$ for $\eta\in E_n$ and $\zeta\in E_m$. It follows that $a:=\sum_{\eta\in E_n}n(\eta)\eta_0^{-1}\eta$ and $b:=\sum_{\zeta\in E_m}m(\zeta)\zeta\eta_0$ belong to the group-ring  $R(c^{-1}(e)\cap\G_{x}^{x})$. Since 
	\begin{equation*}
		\sum_{\eta\in E_n,\ \zeta\in E_m,\ \eta\zeta=x}m(\zeta)n(\eta)
		=m1_{U_x}n(x)=1
	\end{equation*}
	and 
	\begin{equation*}
		\sum_{\eta\in E_n,\ \zeta\in E_m,\ \eta\zeta=\rho}m(\zeta)n(\eta)
		=m1_{U_x}n(\rho)=0
	\end{equation*}
	for $\rho\in \G_{x}^{x}\setminus\{x\}$, it follows that 
	\begin{equation*}
		ba=\left(\sum_{\zeta\in E_m}m(\zeta)\zeta\eta_0\right) \left(\sum_{\eta\in E_n}n(\eta)\eta_0^{-1}\eta\right)
		=\sum_{\eta\in E_n,\ \zeta\in E_m}m(\zeta)n(\eta)\zeta\eta=x.
	\end{equation*}
	Thus, $b$ is the inverse of $a$ in $R(c^{-1}(e)\cap\G_{x}^{x})$. Since $R(c^{-1}(e)\cap\G_{x}^{x})$ only has trivial units, it follows that there is an $\eta_n\in E_n$ and an $\zeta_m\in E_m$ such that $n(\eta)=m(\zeta)=0$ for all $\eta\in E_n\setminus\{\eta_n\}$ and all $\zeta\in E_m\setminus\{\zeta_m\}$, and $\eta_n^{-1}=\zeta_m$. Thus, $\eta_n$ is the only $\eta\in\supp(n)$ such that $s(\eta)=x$ and $r(\eta)\in U_x$, and  $\zeta_m$ is the only $\zeta\in\supp(m)$ such that $r(\zeta)=x$ and $s(\eta)\in U_x$. A similar argument, using that $R(c^{-1}(e)\cap\G_{x}^{x})$ has no zero-divisors, shows that if $U\in F\setminus\{U_x\}$, then there is no $\eta\in\supp(n)$ such that $s(\eta)=x$ and $r(\eta)\in U$, and no $\zeta\in\supp(m)$ such that $r(\zeta)=x$ and $s(\eta)\in U$.
	
	We have shown that if $x\in r(\supp(m))=s(\supp(n))$ and the group-ring $R(c^{-1}(e)\cap\G_{x}^{x})$ has no zero-divisors and only trivial units, then there is a unique $\eta\in\supp(n)$ such that $s(\eta)=n$, and a unique $\zeta\in\supp(m)$ such that $r(\zeta)=x$. Since there is a dense subset $X \subseteq \G^0$ such that the group-ring $R(c^{-1}(e)\cap\G_x^x)$ has no zero-divisors and only trivial units for all $x \in X$, and since $\supp(m)$ and $\supp(n)$ are each the union of a finite number of compact open bisections, it follows that for every $x\in r(\supp(m))=s(\supp(n))$ there is a unique $\eta\in\supp(n)$ such that $s(\eta)=n$, and a unique $\zeta\in\supp(m)$ such that $r(\zeta)=x$. That is, the restriction of $r$ to $\supp(m)$ and the restriction of $s$ to $\supp(n)$ are injective. In a similar way one can show that the restriction of $s$ to $\supp(m)$ and the restriction of $r$ to $\supp(n)$ are injective. Thus, $\supp(m)$ and $\supp(n)$ are compact open bisections contained in $c^{-1}(g)$ and $c^{-1}(g^{-1})$, respectively.
	
	(3): Since $mn=1_{s(\supp(n))}=1_{r(\supp(m))}$ and $nm=1_{s(\supp(m))}=1_{r(\supp(n))}$, and since $\supp(m)$ and $\supp(n)$ are bisections, it follows that $\supp(n)=\supp(m)^{-1}$.
	\end{proof}
	
\subsection{The groupoid $W_\G$} For a commutative integral domain $R$ with identity, a discrete group $\Gamma$, and $(\G,c)\in\cls$, we construct from $A_R(\G)$ and $D_R(\G)$, a graded groupoid $W_\G$ which is isomorphic to $\G$.

If $A \subseteq \G$ is an open bisection, then we let $\alpha_A$ denote the homeomorphism from $s(A)$ to $r(A)$ given by $\alpha_A(s(\eta))=r(\eta)$ for $\eta\in A$.

\begin{proposition}\label{pro:1}
	Let $R$ be a commutative integral domain with identity, $\Gamma$ a discrete group, and $(\G,c)\in\cls$.
	\begin{enumerate}
		\item\label{pro:1equivrel} There is an equivalence relation $\sim$ on \[ \{((m,n),g,x):g \in \Gamma, (m,n)\in N_R(\G)_g,\ x \in s(\supp(m))\} \] such that $((m_1,n_1),g_1,x_1)\sim ((m_2,n_2),g_2,x_2)$ if and only if $x_1=x_2$ and there is a compact open neighbourhood $U\subseteq\G^0$ of $x_1$ such that $m_11_Un_2\in D_R(\G)$.
		\item\label{pro:1bijection} Define \[ W_\G := \{((m,n),g,x):g \in \Gamma, (m,n)\in N_R(\G)_g,\ x \in s(\supp(m))\} / \sim.\] There is a bijection $\psi_\G:W_\G\to\G$ such that $\psi_\G([((m,n),g,x)])\in\supp(m)\cap s^{-1}(\{x\})$ for $g\in G$, $(m,n)\in N_R(\G)_g$ and $x \in s(\supp(m))$, and such that 
$$\psi_\G([((1_A,1_{A^{-1}}),c(\eta),s(\eta))])=\eta$$ 
for $\eta\in\G$ and any compact open bisection $A$ that contains $\eta$.
		\item\label{pro:1inverse} Let $[((m,n),g,x)] \in W_\G$. Then \[ \psi_\G([((m,n),g,x)])^{-1}=\psi_\G([(n,m),g^{-1},\alpha_{\supp(m)}(x))]).\]
		\item\label{pro:1mult} For $i=1,2$ let $[((m_i,n_i),g_i,x_i)] \in W_\G$ and  $\eta_i:=\psi_\G([((m_i,n_i),g_i,x_i)])$. Then $\eta_1$ and $\eta_2$ are composable if and only if $x_1=\alpha_{\supp(m_2)}(x_2)$, in which case $\eta_1 \eta_2=\psi_\G([((m_1m_2,n_2n_1),g_1g_2,x_2)]).$
		\item\label{pro:1cocycle} Let $[((m,n),g,x)] \in W_\G$. Then $c(\psi_\G([((m,n),g,x)]))=g$.
		\item\label{pro:1preimage} Let $g\in\Gamma$ and suppose that $A\subseteq c^{-1}(g)$ is a compact open bisection. Then $\psi_\G^{-1}(A)=\{[((m,n),g,x)] \in W_\G:\supp(m)\subseteq A\}$.
	\end{enumerate}
\end{proposition}

\begin{proof}
\eqref{pro:1equivrel} Define $\sim$ by $((m_1,n_1),g_1,x_1) \sim ((m_2,n_2),g_2,x_2)$ if and only if $x_1=x_2$ and there is a compact open neighbourhood $U\subseteq\G^0$ of $x_1$ such that 
$$\supp(m_1) \bigcap s^{-1}(U)=\supp(m_2) \bigcap s^{-1}(U).$$ 
It is easy to check that $\sim$ is an equivalence relation on 
\[ \{((m,n),g,x):g \in \Gamma, (m,n)\in N_R(\G)_g, x \in s(\supp(m))\}. \]

Suppose $(m_1,n_1), (m_2,n_2) \in N_R(\G)$, $x\in s(\supp(m_1))\cap s(\supp(m_2))$ and $U\subseteq\G^0$ is a compact open neighbourhood of $x$. Then $\supp(m_1)$ and $\supp(n_2)$ are bisections and $\supp(n_2)=\supp(m_2)^{-1}$ by Lemma~\ref{lem:1}. Since
\begin{equation*}
	m_11_Un_2(\eta)=\sum_{\begin{subarray}{l}s(\eta_1)=r(\eta_2)\in U\\ \eta_2\eta_1=\eta\end{subarray}}m_1(\eta_1)n_2(\eta_2)
\end{equation*}
for all $\eta\in\G$, it follows that $m_11_Un_2\in D_R(\G)$ if and only if $\supp(m_1) \bigcap s^{-1}(U)=\supp(m_2) \bigcap s^{-1}(U)$.


\eqref{pro:1bijection} Let $(m,n) \in N_R(\G)$ and $x \in s(\supp(m))$. Then there is a unique $\eta_{(m,n),x}\in\supp(m)$ such that $s(\eta_{(m,n),x})=x$. It is clear that if $((m_1,n_1),g_1,x_1) \sim ((m_2,n_2),g_2,x_2)$, then $\eta_{(m_1,n_1),x_1}=\eta_{(m_2,n_2),x_2}$. Thus, the map $\psi_\G: W_\G \to \G$ defined by $[((m,n),g,x)]\mapsto\eta_{(m,n),x}$ is a well-defined map. 

Suppose $\eta:=\psi_\G([((m_1,n_1),g_1,x_1)])=\psi_\G([((m_2,n_2),g_2,x_2)])$. Then $x_1=s(\eta)=x_2$ and $U:=\{\eta_1^{-1}\eta_2:\eta_1\in\supp(m_1),\ \eta_2\in\supp(m_2)\}\cap\G^0$ is a compact open neighbourhood of $x_1$ such that $\supp(m_1) \bigcap s^{-1}(U)=\supp(m_2) \bigcap s^{-1}(U)$, so $((m_1,n_1),g_1,x_1) \sim ((m_2,n_2),g_2,x_2)$. This shows that $\psi_\G$ is injective.



Suppose $\eta\in\G$. Let $A \subseteq c^{-1}(c(\eta))$ be a compact open bisection containing $\eta$. Then $\psi_\G([((1_A,1_{A^{-1}}),c(\eta),s(\eta))])=\eta$. This shows that $\psi_\G$ is also surjective.

\eqref{pro:1inverse} Let $\eta:=\psi_\G([((m,n),g,x)])$. Then $x = s(\eta)$ and $\alpha_{\supp(m)}(x)=r(\eta)$, so 
\begin{align*} 
	\psi_\G([((n,m), g^{-1},\alpha_{\supp(m)}(x))]) &= \psi_\G([((n,m),g^{-1},r(\eta))]) \\ &=\psi_\G([((n,m),g^{-1},s(\eta^{-1}))]) \\
&=\eta^{-1}.
\end{align*}

\eqref{pro:1mult} We have that $\eta_1,\eta_2 \in \G$ are composable if and only if $x_1 = s(\eta_1)=r(\eta_2)=\alpha_{\supp(m_2)}(x_2)$. We also have that 
$$\eta_1\eta_2\in\supp(m_1m_2) \subseteq \supp(m_1)\supp(m_2) \subseteq c^{-1}(g_1)c^{-1}(g_2)=c^{-1}(g_1 g_2),$$ 
and that $x_2 = s(\eta_2) =s(\eta_1\eta_2)$. It follows that $[((m_1m_2,n_2n_1),g_1g_2,x_2)] \in W_\G$ and $\psi_\G([((m_1m_2,n_2n_1),g_1g_2,x_2)])=\eta_1 \eta_2$.

\eqref{pro:1cocycle} Let $\eta:=\psi_\G([((m,n),g,x)])$. Then $\eta\in\supp(m) \subseteq c^{-1}(g)$, so 
$$c(\psi_\G([((m,n),g,s(\eta))]))= c(\eta) =g.$$

\eqref{pro:1preimage} Since $\psi_\G([((m,n),g,x)])\in\supp(m)$, it follows that $\psi_\G(\{[((m,n),g,x)] \in W_\G:\supp(m)\subseteq A\})\subseteq A$. Conversely, if $\eta\in A$, then $\psi_\G([(1_A,1_A^{-1}),g,s(\eta)])=\eta$, and since $\supp(1_A)=A$, this shows that $A\subseteq \psi_\G(\{[((m,n),g,x)] \in W_\G:\supp(m)\subseteq A\})$.
\end{proof}

\begin{remark}
If follows from Proposition~\ref{pro:1} that if we equip $W_\G$ with a partially-defined product given by 
$$[((m_1,n_1),g_1,x_1)][((m_2,n_2),g_2,x_2)]=[((m_1m_2,n_2n_1),g_1g_2,x_2)]$$ 
if $x_1=\alpha_{\supp(m_2)}(x_2)$, an inverse operation given by 
$$[((m,n),g,x)]^{-1}=[(n,m),g^{-1},\alpha_{\supp(m)}(x)],$$ 
a $\Gamma$-grading 
$$[((m,n),g,x)]\mapsto g,$$ 
and the topology generated by subsets of the form $\{[((m,n),g,x)] \in W_\G:\supp(m)\subseteq A\}$ where $A\subseteq c^{-1}(g)$ is a compact open bisection, then $W_\G$ is a $\Gamma$-graded ample Hausdorff groupoid which is isomorphic to $\G$.
\end{remark}

\subsection{A homeomorphism of unit spaces}

We now generalise \cite[Proposition 5.2]{BCH} and show that a diagonal-preserving ring-isomorphism of Steinberg algebras induces a homeomorphism of the unit spaces of the groupoids.

Let $R$ be a commutative integral domain with identity and $\G$ an ample Hausdorff groupoid. Then $B(\G) := \{ p \in D_R(\G): p^2 = p\}$ is a Boolean algebra with $p \vee q = p + q - pq$, $p \wedge q = pq$ and $p \le q$ if and only if $pq = p$. Recall that a \emph{filter} of a Boolean algebra $B$ is a subset $U$ of $B$ such that $p,q\in U\implies p\wedge q\in U$ and $p\in U$, $p\le q\implies q\in U$, and that an \emph{ultra filter} is a filter $U$ that does not contain the zero element of $B$ and has the property that $U\subseteq U'\implies U=U'$ for any other filter $U'$ that does not contain the zero element. Let $B(\G)^* := \{U \subseteq B(\G): U \text{ is an ultrafilter} \}$ equipped with the topology generated by the subbasic open sets $N_p := \{ U \in B(\G)^*: p \in U\}$ for $p \in B(\G)$. 

\begin{proposition}
Let $R$ be a commutative integral domain with identity and let $\G$ be an ample Hausdorff groupoid. Define $\rho_\G: \G^0 \to B(\G)^*$ by \[\rho_\G(x) := \{1_L \in D_R(\G): L \text{ is a compact open neighbourhood of } x \}.\] Then $\rho_\G$ is a homeomoprhism with inverse $\rho_\G^{-1}$ taking $U \in B(\G)^*$ to the unique element of $\bigcap_{1_L \in U} L$.
\end{proposition}

\begin{proof}
Our proof follows the proof of \cite[Proposition~5.1]{BCH}. Since $\G$ is ample, the unit space has a basis consisting of compact open sets, so each unit has a neighbourhood basis consisting of compact open neighbourhoods.

Fix $x \in \G^0$. It is routine to check that $\rho_\G(x)$ is a filter. We check that it is an ultrafilter. Suppose $U$ is a filter such that $0 \not \in U$ and $\rho_\G(x) \subseteq U$. It suffices to show that $U \subseteq \rho_\G(x)$. Suppose $p \in U$ and let $C:=\supp(p)$. Then $1_{C\cap L}=p\wedge 1_L\in U$ and $C \cap L \ne \emptyset$ for any compact open neighbourhood $L$ of $x$. Thus $x \in C$ because $C$ is closed, so $p=1_C \in \rho_\G(x)$.

Since compact open sets separate points in $\G^0$, it follows that $\rho_\G$ is injective.
Arguments identitcal to the ones used in the proof of \cite[Proposition~5.1]{BCH} show that for $U \in B(\G)^*$, the intersection $\bigcap_{1_L \in U} L$ is a singleton $\{x\} \subseteq \G^0$, that $\rho_\G$ is surjective with inverse given by $\rho_\G^{-1}(U)=x$, and that $\rho_\G$ is continuous.

It remains to check that $\rho_\G$ is open: if $C$ is a compact open subset of $\G^0$, then $\rho_\G(C) = \{ U \in B(\G)^*: 1_C \in U \} = N_{1_C}$ is open. \end{proof}

\begin{proposition}\label{pro:2}
	Let $R$ be a commutative integral domain with identity and let $\G_1$ and $\G_2$ be ample Hausdorff groupoids. If $\phi:A_R(\G_1)\to A_R(\G_2)$ is a diagonal-preserving ring-isomorphism, then there is a homeomorphism $\phi^*:\G_2^0\to\G_1^0$ such that $f(\phi^*(x))=\phi(f)(x)$ for $f\in D_R(\G_1)$ and $x\in\G_2^0$.
\end{proposition}

\begin{proof}
The isomorphism $\phi$ restricts to an isomorphism from the Boolean algebra $B(\G_1)$ onto $B(\G_2)$, and induces a homeomorphism $\tilde \phi: B(\G_2)^* \to B(\G_1)^*$ given by $\tilde \phi(U) := \phi^{-1}(U)$ for $U \in B(\G_2)^*$. Define $\phi^* := \rho_{\G_1}^{-1} \circ \tilde \phi \circ \rho_{\G_2}$. Then $\phi^*$ is a homeomorphism from $\G_2$ onto $\G_1$ such that $\phi^*(x)$ is the unique element of $\bigcap_{1_M \in \tilde \phi (\rho_{\G_2}(x))} M$.

Arguments identical to the ones used to prove \cite[Proposition~5.2 and Lemma~5.3]{BCH} show that $\phi^*$ satisfies the desired property.
\end{proof}

\subsection{Proof of Theorem~\ref{thm:steinberg}} To prove Theorem~\ref{thm:steinberg}, we use Proposition~\ref{pro:2} to show that a diagonal-preserving graded ring-isomorphism of Steinberg algebras $A_R(\G_1) \cong A_R(\G_2)$ induces a graded isomorphism of the groupoids $W_{\G_1}$ and $W_{\G_2}$.

\begin{proposition}\label{pro:3}
	Let $R$ be a commutative integral domain with identity, $\Gamma$ a discrete group, and $(\G_1,c_1), (\G_2,c_2) \in \cls$. If $\phi:A_R(\G_1)\to A_R(\G_2)$ is a diagonal-preserving graded ring-isomorphism, then there is a $\Gamma$-graded isomorphism $\Phi: W_{\G_1} \to W_{\G_2}$ such that 
$$\Phi([((m,n),g,x)])=[((\phi(m),\phi(n)),g,(\phi^*)^{-1}(x))]$$ 
for $g\in\Gamma$, $(m,n)\in N_R(\G_1)_g$ and $x\in\supp(m)$.
\end{proposition}

\begin{proof}
Our proof follows the proof of \cite[Theorem~6.1]{BCH}. Let $g\in\Gamma$, $(m,n)\in N_R(\G_1)_g$ and $x\in\supp(m)$. Then $(\phi(m),\phi(n))\in N_R(\G_2)_g$ since $\phi$ is a diagonal-preserving graded ring-isomorphism. Since $(\phi^*)^{-1}(s(\supp(m)))=s(\supp(\phi(m)))$, it follows that \linebreak$(\phi^*)^{-1}(x)\in s(\supp(\phi(m)))$ and thus that $[((\phi(m),\phi(n)),g,(\phi^*)^{-1}(x))]\in W_{\G_2}$.  


To see $\Phi$ is well-defined suppose $((m_1, n_1), g_1, x_1) \sim ((m_2,n_2),g_2,x_2)$. Then $x_1 = x_2$ and there is a compact open neighbourhood $U \subseteq \G_1$ of $x_1$ such that $m_1 1_U n_2 \in D_R(\G_1)$. Then $V :=(\phi^*)^{-1}(U) \subseteq \G_2^0$ is a compact open neighbourhood of $\phi(x_1)$ and $1_V(x) = 1_U(\phi^*(x)) = \phi(1_U)(x)$ for $x \in \G_2^0$ by Proposition~\ref{pro:2}, so $\phi(m_1) 1_V \phi(n_2) = \phi(m_1 1_U n_2) \in D_R(\G_2)$. Hence $((\phi(m_1), \phi(n_1)), g_1, (\phi^*)^{-1}(x_1)) \sim ((\phi(m_2), \phi(n_2)), g_2, (\phi^*)^{-1}(x_2)),$ so $\Phi$ is well-defined.

The map $[((m,n),g,x)] \mapsto [((\phi^{-1}(m),\phi^{-1}(n)),g,\phi^*(x))]$ is clearly an inverse for $\Phi$, so $\Phi$ is a bijection.

The topology on $W_\G$ is generated by basic open sets $Z((m,n),g,A) := \{[((m,n),g,x)]: \supp(m) \subseteq A \}$ for $g \in \Gamma$, $(m,n) \in N_R(\G)_g$ and $A \subseteq c^{-1}(g)$. We have $\Phi(Z((m,n),g, A)) = Z((\phi(m),\phi(n)),g, (\phi^*)^{-1}(A))$, so $\Phi$ is continous. We also have $\Phi^{-1}(Z((m,n),g,A))=Z((\phi^{-1}(m),\phi^{-1}(n)),g,\phi^*(A))$, so $\Phi$ is open. Hence $\Phi$ is a homeomorphism.

Finally, we check that $\Phi$ is a homomorphism. We have that $m 1_U n(\alpha_{\supp(m)}(x))=1_U(x)$ for any compact open subset $U$ of $s(\supp(m))$, so \[(\phi^*)^{-1}(\alpha_{\supp(m)}(x)) = \alpha_{\supp(\phi(m))}((\phi)^*)^{-1}(x).\] Using this at the fourth equality, we calculate
\begin{align*}
& \Phi([((m_1,n_1),g_1,\alpha_{\supp(m_2)}(x_2))][((m_2,n_2),g_2, x_2)]) \\
&\quad = \Phi([((m_1m_2,n_2 n_1), g_1 g_2, x_2)]) \\
&\quad =([((\phi(m_1 m_2), \phi(n_1 n_2)), g_1 g_2, (\phi^*)^{-1}(x_2)]) \\
&\quad = [((\phi(m_1), \phi(n_1)), g_1, \alpha_{\supp(\phi(m_2))}((\phi^*)^{-1}(x_2)))][((\phi(m_2), \phi(n_1)), g_2, (\phi^*)^{-1}(x_2))] \\
&\quad = [((\phi(m_1), \phi(n_1)), g_1, (\phi^*)^{-1}(\alpha_{\supp(m_2)}(x_2)))][((\phi(m_2), \phi(n_1)), g_2, (\phi^*)^{-1}(x_2))] \\
&\quad = \Phi([((m_1, n_1), g_1, \alpha_{\supp(m_2)}(x_2))])\Phi([((m_2, n_1), g_2, x_2)]). 
\end{align*}
\end{proof}

\begin{proof}[Proof of Theorem~\ref{thm:steinberg}]

$(1)\implies (2)$: Suppose that $\psi: \G_1 \to \G_2$ is an isomorphism such that $c_2 \circ \psi = c_1$. Then there is a graded $*$-algebra-isomorphism $\phi: A_R(\G_1) \to A_R(\G_2)$ defined by $\phi(f) = f \circ \psi^{-1}$. We have that $\phi(D_R(\G_1)) = D_R(\G_2)$ because $\psi(\G_1^0)=G_2^0$.

$(2)\implies (3)$: Any diagonal-preserving graded $*$-algebra-isomorphism between $A_R(\G_1)$ and $A_R(\G_2)$ is also a diagonal-preserving graded ring-isomorphism between $A_R(\G_1)$ and $A_R(\G_2)$.

$(3)\implies (1)$: Suppose that $\phi:A_R(\G_1) \to A_R(\G_2)$ is a ring-isomorphism satisfying $\phi(D_R(\G_1))=D_R(\G_2)$. Let $\Phi: W_{\G_1}\to W_{\G_2}$ be the isomorphism of Proposition~\ref{pro:3} and $\psi_{\G_1}: W_{\G_1} \to \G_1$ and $\psi_{\G_2}: W_{\G_2} \to \G_2$ be the isomorphisms from Proposition~\ref{pro:1}. Then $\psi := \psi_{\G_2} \circ \Phi \circ \psi_{\G_1}^{-1}$ is an isomorphism from $\G_1$ onto $\G_2$ such that $c_2 \circ \psi = c_1$.
\end{proof}

\subsection{Diagonal-preserving stable isomorphisms of Steinberg algebras}

We present a ``stabilised version" of Theorem~\ref{thm:steinberg} and an analogue of \cite[Corollary 4.5]{CRS}.

As in \cite{CRS}, we write $\mathcal{R}$ for the full countable equivalence relation $\N\times\N$ regarded as a discrete principal groupoid with unit space $\N$. If $\G$ is an ample Hausdorff groupoid, then the product groupoid $\G\times\mathcal{R}$ is also ample and Hausdorff. If $c$ is a continuous cocycle from $\G$ to a discrete group $\Gamma$, then we let $\bar{c}$ denote the continuous cocycle from $\G\times\mathcal{R}$ to $\Gamma$ given by $\bar{c}(\eta,\xi)=c(\eta)$. 

If $R$ is a ring with identity, then we write $M_\infty(R)$ for the ring of finitely supported, countable infinite square matrices over $R$, and $D_\infty(R)$ for the abelian subring of $M_\infty(R)$ consisting of diagonal matrices. For ample Hausdorff groupoids $\G_1$ and $\G_2$, we say that an isomorphism $\phi: A_R(\G_1) \otimes M_\infty(R) \to A_R(\G_2) \otimes M_\infty(R)$ is diagonal-preserving if $\phi(D_R(\G_1) \otimes D_\infty(R)) \to D_R(\G_2) \otimes D_\infty(R)$.

If there is a $\Gamma$-grading $\bigoplus_{g \in \Gamma} A_R(\G)_g$ of $A_R(\G)$, then we get a $\Gamma$-grading of $A_R(\G)\otimes M_\infty(R)$ by setting $(A_R(\G)\otimes M_\infty(R))_g:=A_R(\G)_g\otimes M_\infty(R)$ for $g\in\Gamma$.

\begin{theorem}\label{thm:stable}
Let $R$ be a commutative integral domain with identity, let $\Gamma$ be a discrete group, and let $(\G_1,c_1), (\G_2,c_2)\in\cls$. The following are equivalent.
\begin{enumerate}
	\item There is an isomorphism $\psi:\G_1\times\mathcal{R}\to\G_2\times\mathcal{R}$ such that $\bar{c}_2\circ\psi=\bar{c}_1$.
	\item There is a diagonal-preserving graded $*$-algebra-isomorphism between \linebreak$A_R(\G_1)\otimes M_\infty(R)$ and $A_R(\G_2)\otimes M_\infty(R)$.
	\item There is a diagonal-preserving graded ring-isomorphism between $A_R(\G_1)\otimes M_\infty(R)$ and $A_R(\G_2)\otimes M_\infty(R)$.
\end{enumerate}
\end{theorem}

\begin{proof}
There is a diagonal-preserving graded $*$-algebra-isomorphism between $A_R(\G_i\times\mathcal{R})$ and $A_R(\G_i)\otimes M_\infty(R)$ for $i \in \{1,2\}$. The result therefore follows from Theorem~\ref{thm:steinberg}.
\end{proof}

If $X$ is a subset of $\G^0$, then we let $\G|_X:=s^{-1}(X)\cap r^{-1}(X)$, and say that $X$ is \emph{$\G$-full} (or just \emph{full}) if $r(s^{-1}(X))=\G^0$. Two ample groupoids $\G_1$ and $\G_2$ are \emph{Kakutani equivalent} if, for $i=1,2$, there is a $\G_i$-full clopen subset $X_i\subseteq\G_i^0$ such that $\G_1|_{X_1}$ and $\G_2|_{X_2}$ are isomorphic (see \cite{CRS} and \cite{Matui}). We say that $\G_1$ and $\G_2$ are \emph{groupoid equivalent} if there is a $\G_1$--$\G_2$ equivalence in the sense of \cite[Definition 2.1]{MRW}. For a ring $A$, we denote by $M(A)$ the \emph{multiplier ring} of $A$ (see for example \cite{AP}). 

\begin{corollary}\label{cor:stable}
Let $R$ be a commutative integral domain with identity. For $i=1,2$, let $\G_i$ be an ample Hausdorff groupoid such that $\G_i^0$ is $\sigma$-compact and such that there is a dense subset $X_i \subseteq \G_I^0$ such that the group-ring $R(c_i^{-1}(e)\cap (\G_i)_x^x)$ has no zero-divisors and only trivial units for all $x \in X_i$. The following are equivalent.
\begin{enumerate}
	\item The groupoids $\G_1$ and $\G_2$ are Kakutani equivalent.
	\item There are full open sets $X_i\subseteq\G_i^0$ such that $(\G_1)|_{X_1}$ and $(\G_2)|_{X_2}$ are isomorphic.
	\item The groupoids $\G_1$ and $\G_2$ are groupoid equivalent.
	\item The groupoids $\G_1\times\mathcal{R}$ and $\G_2\times\mathcal{R}$ are isomorphic.
	\item There is a diagonal-preserving $*$-algebra-isomorphism between $A_R(\G_1)\otimes M_\infty(R)$ and $A_R(\G_2)\otimes M_\infty(R)$.
	\item There is a diagonal-preserving ring-isomorphism between $A_R(\G_1)\otimes M_\infty(R)$ and \linebreak $A_R(\G_2)\otimes M_\infty(R)$.
	\item There are projections $p_i\in M(D_R(\G_i))$ such that $p_i$ is full in $A_R(\G_i)$, and a $*$-algebra-isomorphism $\phi:p_1A_R(\G_1)p_1\to p_2A_R(\G_2)p_2$ such that $\phi(p_1D_R(\G_1))=p_2D_R(\G_2)$.
	\item There are idempotents $p_i\in M(D_R(\G_i))$ such that $p_i$ is full in $A_R(\G_i)$, and a ring-isomorphism $\phi:p_1A_R(\G_1)p_1\to p_2A_R(\G_2)p_2$ such that $\phi(p_1D_R(\G_1))=p_2D_R(\G_2)$.
\end{enumerate}
\end{corollary}

\begin{proof}
The equivalence of (1)--(4) is proved in \cite[Theorem 3.2]{CRS}, and the equivalence of (4), (5), and (6) follows from Theorem~\ref{thm:stable} by letting $\Gamma$ be the trivial group. That (1), (7), and (8) are equivalent can be proved in a similar way to \cite[Corollary 4.5]{CRS}.
\end{proof}

\begin{remark}
It is natural to ask if there is a ``graded version" of Corollary~\ref{cor:stable}. We have not been able to prove such a result.
\end{remark}

\subsection{Diagonal-preserving isomorphisms of Steinberg algebras and actions of inverse semigroups}

We end this section with a corollary which might be useful if one wants to obtain results similar to those in the next section for systems, other than directed graphs, from which one can construct an ample Hausdorff groupoid.

Recall that if $\G$ is an ample Hausdorff groupoid, and $A$ is a compact open bisection of $\G$, then we denote by $\alpha_A$ the homeomorphism from $s(A)$ to $r(A)$ given by $\alpha_A(s(\eta))=r(\eta)$ for $\eta\in A$. If $c$ is a cocycle from $\G$ to a discrete group $\Gamma$, then we let $S_\G^c$ denote the inverse semigroup of compact open bisections $A$ of $\G$ such that $c(A)$ is a singleton (such sets are called \emph{homogeneous} in \cite{ABHS}). We let $\tilde{c}:S_\G^c\to \Gamma$ be the map given by $\tilde{c}(A)=g$ for $A\in S_\G^c$ with $A\subseteq c^{-1}(g)$.

For $A\in S_\G^c$ and $x\in s(A)$, let $\eta_{(A,x)}$ denote the unique element $\eta\in A$ for which $s(\eta)=x$. Then the map $A\mapsto\eta_{(A,x)}$ induces a group isomorphism from $\{A\in S_\G^c: x\in s(A),\ \alpha_A(x)=x,\ c(A)=\{e\}\}/\approx$ to $c^{-1}(e)\cap(\G)_x^x$ where $\approx$ is the equivalence relation given by $A_1\approx A_2$ if there is a compact open neighbourhood $U$ of $x$ in $\G^0$ such that $\{\eta\in A_1:s(\eta)\in U\}=\{\eta\in A_2:s(\eta)\in U\}$.
It is therefore possible to decide if $(\G,c)\in\cls$ from $S_\G^c$ and its action on $\G^0$ alone.

\begin{corollary}\label{cor:action}
	Let $R$ be a commutative integral domain with identity, let $\Gamma$ be a discrete group, and let $(\G_1,c_1), (\G_2,c_2)\in\cls$. The following are equivalent.
	\begin{enumerate}
		\item There is a diagonal-preserving graded $*$-algebra-isomorphism between $A_R(\G_1)$ and $A_R(\G_2)$.
		\item There is a diagonal-preserving graded ring-isomorphism between $A_R(\G_1)$ and \linebreak$A_R(\G_2)$.
		\item There is an isomorphism $\psi:\G_1\to\G_2$ such that $c_2\circ\psi=c_1$.
		\item There is an inverse semigroup isomorphism $\tau:S_{\G_1}^{c_1}\to S_{\G_2}^{c_2}$ such that $\tilde{c}_1=\tilde{c}_2\circ\tau$, and a homeomorphism $h:\G_1^0\to\G_2^0$ such that $s(\tau(A))=h(s(A))$, $r(\tau(A))=h(r(A))$, and $\alpha_{\tau(A)}(h(x))=h(\alpha_A(x))$ for $A\in S_{\G_1}^{c_1}$ and $x\in s(A)$.
	\end{enumerate}
\end{corollary}

\begin{proof}
The equivalence of (1), (2), and (3) is proved in Theorem~\ref{thm:steinberg}. It is routine to check that (3) implies (4).
It follows from \cite[Proposition 5.4]{Exel} that the groupoid of germs for the action $A\mapsto\alpha_A$ of $S_{\G_i}^{c_i}$ on $\G_i^0$ is isomorphic to $\G_i$. It is straightforward to check that the map $\tilde{c}:S_{\G_i}^{c_i}\to \Gamma$ induces a cocycle on the groupoid of germs and that the isomorphism from \cite[Proposition 5.4]{Exel} intertwines this cocycle with $c_i$. It follows that (4) implies (3).
\end{proof}

\section{Diagonal-preserving graded isomorphisms of graph algebras}\label{sec:leavitt}

In this section, we apply Theorem~\ref{thm:steinberg} to graph groupoids in order to strengthen \cite[Corollary 4.4]{ABHS}, \cite[Theorem 5.3]{AER}, \cite[Theorem 6.2]{BCH}, \cite[Corollary 6.1]{CEOR}, and \cite[Corollary 6.4]{CEOR}, and obtain Leavitt path algebra versions of the results of \cite{CR}.

\subsection{Diagonal-preserving graded isomorphisms of graph algebras}
For a directed graph $E$ and a unital ring $R$, we let $\G_E$ denote the groupoid of $E$ (see for example \cite{AER,BCW,Car,CR,CW,CS}), $C^*(E)$ the $C^*$-algebra of $E$ (see for example \cite{AT,AER,BCW,Car,CR,Sor,Web}), $\mathcal{D}(E)$ the $C^*$-subalgebra $\overline{\operatorname{span}}\{s_\mu s_\mu^*:\mu\in E^*\}$ of $C^*(E)$, $L_R(E)$ the Leavitt path $R$-algebra of $E$ (see for example \cite{BCH,Car,Tomforde2}), and $D_R(E)$ for the $*$-subalgebra $\operatorname{span}_R\{\mu \mu^*:\mu\in E^*\}$ of $L_R(E)$. We say that an isomorphism between two Leavitt path algebras $L_R(E)$ and $L_R(F)$ is \emph{diagonal-preserving} if it maps $D_R(E)$ onto $D_R(F)$, and that a $*$-isomorphism between two graph algebras $C^*(E)$ and $C^*(F)$ is \emph{diagonal-preserving} if it maps $\mathcal{D}(E)$ onto $\mathcal{D}(F)$.

Recall that we are assuming that $R$ comes with an involution $r\mapsto r^*$ that is also a ring automorphism (the involution could be the identity map). The involution extends to an involution on $L_R(E)$ given by $r\mu\nu^*\mapsto r^*\nu\mu^*$ for $r\in R$ and $\mu,\nu\in E^*$, making $L_R(E)$ a $*$-algebra. 
A \emph{unital $*$-subring of} $\C$ is a subring $K$ of $\C$ that is closed under complex conjugation and contains $1$. The involution on $K$ is then given by complex conjugation. 
As in \cite{Car}, we say that a unital $*$-subring $K$ of $\C$ is \emph{kind} if $\lambda_0,\lambda_1,\dots,\lambda_n\in K$ satisfying $\lambda_0=\sum_{i=0}^n|\lambda_i|^2$ implies that $\lambda_1=\dots=\lambda_n=0$. For example, $\mathbb{Z}$ is a kind unital $*$-subring of $\C$. According to  \cite[Corollary~6]{Car}, if $K$ is a kind unital $*$-subring $K$ of $\C$, then any $*$-ring-isomorphism $L_K(E)\cong L_K(F)$ maps $D_K(E)$ onto $D_K(F)$.

Let $E$ be a directed graph, and let $k:E^1\to\mathbb{R}$ be a function. Then $k$ extends to a function $k:E^*\to\mathbb{R}$ given by $k(v)=0$ for $v\in E^0$ and $k(e_1\dots e_n)=k(e_1)+\dots +k(e_n)$ for $e_1\dots e_n\in E^n$, $n\ge 1$. We then get a continuous cocycle $c_k:\G_E\to\mathbb{R}$ given by $c_k((\mu x,|\mu|-|\nu|,\nu x))=k(\mu)-k(\nu)$. Let $\Gamma_k$ denote the additive subgroup $\Gamma_k=c_k(\G_E)$ of $\mathbb{R}$ generated by $k(E^1)$. For any commutative ring $R$ with identity, we get a $\Gamma_k$-grading of $L_R(E)$ by letting $L_R(E)^k_g:= \operatorname{span}_R \{\mu\nu^*:\mu,\nu\in E^*,\ r(\mu)=r(\nu),\ k(\mu)-k(\nu)=g\}$ for $g\in \Gamma_k$.
We also get a \emph{generalised gauge action} $\gamma^{E,k}:\mathbb{R}\to\operatorname{Aut}(C^*(E))$ given by $\gamma^{E,k}_t(p_v)=p_v$ for $v\in E^0$ and $\gamma^{E,k}_t(s_e)=e^{ik(e)t}s_e$ for $e\in E^1$. 

From Theorem~\ref{thm:steinberg} we get the following strengthening of \cite[Theorem 3.1]{CR}.

\begin{theorem}\label{thm:1}
	Let $E$ and $F$ be directed graphs and $k:E^1\to\mathbb{R}$ and $l:F^1\to\mathbb{R}$ functions. Let $R$ be a commutative integral domain with identity and $K$ a kind unital $*$-subring of $\C$. The following are equivalent.
	\begin{enumerate}
		\item[(1)] There is an isomorphism $\psi:\G_E\to\G_F$ such that $c_l\circ\psi=c_k$.
		\item[(2)] $\Gamma_k=\Gamma_l$ and there is a diagonal-preserving $*$-algebra-isomorphism $\phi:L_R(E)\to L_R(F)$ such that $\phi(L_R(E)^k_g)=L_R(F)^l_g$ for $g\in \Gamma_k$.
		\item[(3)] $\Gamma_k=\Gamma_l$ and there is a diagonal-preserving ring-isomorphism $\phi:L_R(E)\to L_R(F)$ such that $\phi(L_R(E)^k_g)=L_R(F)^l_g$ for $g\in \Gamma_k$.
		\item[(4)] $\Gamma_k=\Gamma_l$ and there is a $*$-ring-isomorphism $\phi:L_K(E)\to L_K(F)$ such that \linebreak$\phi(L_K(E)^k_g)=L_K(F)^l_g$ for $g\in \Gamma_k$.
		\item[(5)] There is a diagonal-preserving $*$-isomorphism $\phi:C^*(E)\to C^*(F)$ such that $\gamma_t^{F,l}\circ\phi=\phi\circ\gamma_t^{E,k}$ for $t\in\mathbb{R}$.
	\end{enumerate}
\end{theorem}

\begin{proof}
(1)$\iff$(2)$\iff$(3): It is shown in \cite[Example 3.2]{CS} that there are diagonal-preserving graded isomorphisms $A_R(\G_E)\cong L_R(E)$ and $A_R(F)\cong L_R(F)$. One easily checks that these isomorphisms are $*$-algebra-isomorphisms. Fix $x \in \G_E^0$. The group $(\G_E)_x^x$ is either trivial or isomorphic to the group of integers, so is indicable throughout. It follows that the group-ring $R(c_k^{-1}(e) \cap (\G_E)_x^x)$ has no zero-divisors and only trivial units by \cite[Theorem~12 and Theorem~13]{Higman}. Similarly, the group-ring $R(c_l^{-1}(e) \cap (\G_F)_x^x)$ has no zero-divisors and only trivial units for all $x \in \G_F^0$. The equivalence of (1), (2), and (3) now follows from Theorem~\ref{thm:steinberg} since $\G_E$ and $\G_F$ are ample Hausdorff groupoids. 

(1)$\iff$(4): It follows from \cite[Corollary~6]{Car} that any $*$-ring-isomorphism $\phi:L_K(E)\cong L_K(F)$ is diagonal-preserving. It therefore follows from (1)$\iff$(3) that (1)$\iff$(4) holds.

(1)$\iff$(5) is proved in \cite[Theorem 3.1]{CR}.
\end{proof}

\subsection{Orbit equivalence of graphs}

For a directed graph $E$, we let $\partial E$ denote the \emph{boundary path space of $E$} (see for example \cite{AER,BCW,Car,CR,CW,Web}), and $\mathcal{P}_E$ the \emph{pseudogroup of $E$} (see \cite{AER,BCW}). Notice that $\partial E$ is homeomorphic to the unit space of $\G_E$. If we identify $\partial E$ and $\G_E^0$ by this homeomorphism, then the pseudogroup $\mathcal{P}_E$ is equal to the pseudogroup $\{\alpha_A:A\text{ is an open bisection of }\G_E\}$. A point $x\in\partial E$ is \emph{eventually periodic} if there are $n,p\in\N$ with $p>0$ such that $\sigma_E^{n+p}(x)=\sigma_E^n(x)$ where $\sigma_E:\partial E^{\ge 1}\to\partial E$ is the shift map (see for example \cite{BCW}). An \emph{orbit equivalence} between two directed graphs $E$ and $F$ is a homeomorphism $h:\partial E\to\partial F$ for which there are continuous maps $k,l:\partial E^{\ge 1}\to\N$ and $k',l':\partial F^{\ge 1}\to\N$ satisfying $\sigma_F^{k(x)}(h(\sigma_E(x)))=\sigma_F^{l(x)}(h(x))$ for all $x\in\partial E^{\ge 1}$ and $\sigma_E^{k'(y)}(h^{-1}(\sigma_F(y)))=\sigma_E^{l'(y)}(h^{-1}(y))$ for all $y\in\partial F^{\ge 1}$ (see \cite[Definition~3.1]{BCW}). A homeomorphism $h:\partial E\to\partial F$ is said to \emph{preserve isolated eventually periodic points} (or just \emph{preserve periodic points}) if $x\in\partial E$ is an isolated eventually periodic point if and only if $h(x)\in\partial F$ is an isolated eventually periodic point (see \cite[Definition 3.2]{AER} and \cite{CW}).

By letting the functions $k:E^1\to\mathbb{R}$ and $l:F^1\to\mathbb{R}$ be constantly equal to 0 and combining Theorem~\ref{thm:1} with \cite[Theorem 5.3]{AER}, we obtain the following strengthening of \cite[Corollary 4.4]{ABHS}, \cite[Theorem 5.3]{AER}, and \cite[Theorem 6.2]{BCH}.

\begin{corollary}\label{cor:orbit}
	Let $E$ and $F$ be directed graphs and let $R$ be a commutative integral domain with identity and $K$ a kind unital $*$-subring of $\C$. The following are equivalent.
	\begin{enumerate}
		\item[(1)] The groupoids $\G_E$ and $\G_F$ are isomorphic.
		\item[(2)] There is a diagonal-preserving $*$-algebra-isomorphism between $L_R(E)$ and \linebreak$L_R(F)$.
		\item[(3)] There is a diagonal-preserving ring-isomorphism between $L_R(E)$ and $L_R(F)$.
		\item[(4)] There is a $*$-ring-isomorphism between $L_K(E)$ and $L_K(F)$.
		\item[(5)] There is a diagonal-preserving $*$-isomorphism between $C^*(E)$ and $C^*(F)$.
	\end{enumerate}
\end{corollary}

\begin{proof}
	The equivalence of (1)--(5) follows from an application of Theorem~\ref{thm:1} with $k:E^1\to\mathbb{R}$ and $l:F^1\to\mathbb{R}$ constantly equal to 0.
\end{proof}

\begin{remark}\label{rmk:orbit} It is shown in \cite[Theorem~5.3]{AER} that conditions (1) and (5) are each equivalent to each of the following two conditions.
\begin{enumerate}
	\item[(6)] There is a homeomorphism $h:\partial E\to\partial F$ such that $\{h\circ\alpha\circ h^{-1}:\alpha\in\mathcal{P}_E\}=\mathcal{P}_F$, and such that if $x\in\partial E$ is isolated, then $x$ is eventually periodic if and only if $h(x) \in\partial F$ is eventually periodic.
	\item[(7)] There is an orbit equivalence from $E$ to $F$ that preserves isolated eventually periodic points.
\end{enumerate}
It follows from \cite[Corollary 4.6]{CW} and the discussion right before \cite[Proposition 3.1]{CW} that if either $E$ and $F$ each satisfy condition (L) (i.e., every cycle in $E$ and every cycle in $F$ have an exit), or $E$ and $F$ each have only finitely many vertices and no sinks, then any homeomorphism $h:\partial E\to\partial F$ automatically preserves isolated periodic points.
\end{remark}

\begin{remark}\label{rmk:isoconjecture}
In \cite[Page~3752]{AT}, Abrams and Tomforde conjectured that if there is a ring-isomorphism between $L_\C(E)$ and $L_\C(F)$, then there is a $*$-isomorphism between $C^*(E)$ and $C^*(F)$. This conjecture was recently confirmed in the case when $E$ and $F$ have finitely many vertices (see \cite[Theorem~14.7]{ERRS}).

By setting $R=\C$ in Corollary~\ref{cor:orbit}, we see that (3) $\implies$ (5) confirms a ``diagonal-preserving" version of Abrams and Tomforde's isomorphism conjecture for arbitrary directed graphs (cf. \cite[Corollary~4.4]{ABHS} and \cite[Corollary~6.3]{BCH}). 
\end{remark}

\subsection{Eventual conjugacy of graphs}
Let $E$ be a directed graph. If the function $k:E^1\to\mathbb{R}$ is constanly equal to 1, then $\Gamma_k=\mathbb{Z}$ and the cocycle $c_k:\G_E\to\mathbb{Z}$ is the standard cocycle $(x,n,y)\mapsto n$ which we denote by $c_E$. The corresponding grading of $L_R(E)$ is the usual $\mathbb{Z}$-grading, where $L_R(E)_n=\{\mu\nu^*:\mu,\nu\in E^*, |\mu|-|\nu|=n\}$, and the corresponding action $\gamma^{E,k}:\mathbb{R}\to\operatorname{Aut}(C^*(E))$ satisfies $\gamma_t^{E,k}=\lambda_{e^{it}}^E$ for all $t\in\mathbb{R}$, where $\lambda^E$ is the standard \emph{gauge action} of $\mathbb{T}$ on $C^*(E)$. 

As in \cite{CR} (see also \cite{MCon}), we say that two directed graphs $E$ and $F$ are \emph{eventually conjugate} if there exist a homeomorphism $h$ from $\partial E$ to $\partial F$ and continuous maps $k:\partial E^{\ge 1}\to\mathbb{N}$ and $k':\partial F^{\ge 1}\to\mathbb{N}$ such that $\sigma_F^{k(x)}(h(\sigma_E(x)))=\sigma_F^{k(x)+1}(h(x))$ for all $x\in \partial E^{\ge 1}$, and likewise $\sigma_E^{k'(y)}(h^{-1}(\sigma_F(y)))=\sigma_E^{k'(y)+1}(h^{-1}(y))$ for all $y\in \partial F^{\ge 1}$. 
We call such a homeomorphism $h: \partial E \to \partial F$ an \emph{eventual conjugacy}.

Notice that if $h:\partial E\to\partial F$ is a conjugacy in the sense that $\sigma_F(h(x))=h(\sigma_E(x))$ for all $x\in \partial E^{\ge 1}$, then $h$ is an eventual conjugacy (in this case we can take $k$ and $k'$ to be constantly equal to 0).

\begin{corollary}\label{thm:3}
	Let $E$ and $F$ be directed graphs and let $R$ be a commutative integral domain with identity and $K$ a kind unital $*$-subring of $\C$. The following are equivalent.
	\begin{enumerate}
		\item[(1)] The graphs $E$ and $F$ are eventually conjugate.
		\item[(2)] There is an isomorphism $\psi:\G_E\to\G_F$ such that $c_F\circ\psi=c_E$.
		\item[(3)] There is a diagonal-preserving $*$-algebra-isomorphism $\phi:L_R(E)\to L_R(F)$ such that \, $\phi(L_R(E)_n)=L_R(F)_n$ for $n\in\mathbb{Z}$.
		\item[(4)] There is a diagonal-preserving ring-isomorphism $\phi:L_R(E)\to L_R(F)$ such that $\phi(L_R(E)_n)=L_R(F)_n$ for $n\in\mathbb{Z}$.
		\item[(4)] There is a $*$-ring-isomorphism $\phi: L_K(E) \to L_K(F)$ such that $\phi(L_K(E)_n)=L_K(F)_n$ for $n \in \mathbb{Z}$.
		\item[(5)] There is a diagonal-preserving $*$-isomorphism $\phi:C^*(E)\to C^*(F)$ such that $\lambda_z^{F}\circ\phi=\phi\circ\lambda_z^{E}$ for $z\in\mathbb{T}$.
	\end{enumerate}
\end{corollary}

\begin{proof}
The equivalence of (1), (2) and (5) was established in \cite[Theorem~4.1]{CR}. The equivalence of (2)--(5) follows from Theorem~\ref{thm:1} by letting $k:E^1\to\mathbb{R}$ and $l:F^1\to\mathbb{R}$ be constantly equal to 1.
\end{proof}

\subsection{Diagonal-preserving graded stable isomorphisms of graph algebras}
Here we present and prove a ``stabilised version'' of Theorem~\ref{thm:1}, giving a strengthening of \cite[Theorem~3.3]{CR}. 

We denote by $\K$ the compact operators on $\ell^2(\N)$, and by $\mathcal{C}$ the maximal abelian subalgebra of $\K$ consisting of diagonal operators. 
We say that an isomorphism $\phi:L_R(E)\otimes M_\infty(R)\to L_R(F)\otimes M_\infty(R)$ is \emph{diagonal-preserving} if $\phi(D_R(E)\otimes D_\infty(R))=D_R(F)\otimes D_\infty(R)$, and that a $*$-isomorphism $\phi:C^*(E)\otimes\mathcal{K}\to C^*(F)\otimes\mathcal{K}$ is \emph{diagonal-preserving} if $\phi:(\mathcal{D}(E)\otimes\mathcal{C})=\mathcal{D}(F)\otimes\mathcal{C}$.

As in \cite{CRS}, for a directed graph $E$, we denote by $SE$ the graph obtained by attaching a head $\dots e_{3,v} e_{2,v} e_{1,v}$ to every vertex $v \in E^0$ (see \cite[Definition 4.2]{T2}). For a function $k:E^1\to\mathbb{R}$, we let $\bar{k}:(SE)^1\to\mathbb{R}$ be the function given by $\bar{k}(e)=k(e)$ for $e\in E^1$, and $\bar{k}(e_{i,v})=0$ for $v\in E^0$ and $i=1,2,\dots$.

\begin{theorem}\label{thm:2}
	Let $E$ and $F$ be directed graphs and $k:E^1\to\mathbb{R}$ and $l:F^1\to\mathbb{R}$ functions. Let $R$ be a commutative integral domain with identity and $K$ a kind unital $*$-subring of $\C$. The following are equivalent.
	\begin{enumerate}
		\item[(1)] There is an isomorphism $\psi:\G_{SE} \to\G_{SF}$ such that $c_{\bar{l}}\circ\psi=c_{\bar{k}}$.
		\item[(2)] $\Gamma_k=\Gamma_l$ and there is a diagonal-preserving $*$-algebra-isomorphism $\phi:L_R(E)\otimes M_\infty(R)\to L_R(F)\otimes M_\infty(R)$ such that $\phi(L_R(E)^k_g\otimes M_\infty(R))=L_R(F)^l_g\otimes M_\infty(R)$ for $g\in \Gamma_k$.
		\item[(3)] $\Gamma_k=\Gamma_l$ and there is a diagonal-preserving ring-isomorphism $\phi:L_R(E)\otimes M_\infty(R)\linebreak\to L_R(F)\otimes M_\infty(R)$ such that $\phi(L_R(E)^k_g\otimes M_\infty(R))=L_R(F)^l_g\otimes M_\infty(R)$ for $g\in \Gamma_k$.
		\item[(4)] $\Gamma_k=\Gamma_l$ and there is a $*$-ring-isomorphism $\phi:L_K(E)\otimes M_\infty(K)\to L_K(F)\otimes M_\infty(K)$ such that $\phi(L_K(E)^k_g\otimes M_\infty(K))=L_K(F)^l_g\otimes M_\infty(K)$ for $g\in \Gamma_k$.
		\item[(5)] There is a diagonal-preserving $*$-isomorphism $\phi:C^*(E)\otimes\mathcal{K}\to C^*(F)\otimes\mathcal{K}$ such that $(\gamma_t^{F,l} \otimes\id_\mathcal{K})\circ\phi=\phi\circ(\gamma_t^{E,k}\otimes\id_\mathcal{K})$ for $t\in\mathbb{R}$.
	\end{enumerate}
	\end{theorem}
	
\begin{proof}
It is shown in \cite[Lemma~4.1]{CRS} that there is an isomorphism $\chi_E:\G_E\times\mathcal{R}\to\G_{SE}$. It is straightforward to check that $c_{\bar{k}}\circ\chi_E=\bar{c_k}$. Similarly, there is an isomorphism $\chi_F:\G_F\times\mathcal{R}\to\G_{SF}$ such that $c_{\bar{l}}\circ\chi_F=\bar{c_l}$. An argument similar to the one used to prove the equivalence of (1)--(4) of Theorem~\ref{thm:1}, using Theorem~\ref{thm:stable} instead of Theorem~\ref{thm:steinberg}, shows that (1)--(4) are equivalent.
	

The equivalence of (1) and (5) is proved in \cite[Theorem~3.3]{CR}.
\end{proof}

\subsection{Two-sided conjugacy and flow equivalence of finite graphs}\label{sec:conjugacy}

For a finite directed graph $E$ with no sinks or sources, we define $\TSS_E$ to be the two-sided edge shift 
$$\TSS_E := \{ (x_n)_{n\in\Z}: x_n \in E^1\text{ and } r(x_n) = s(x_{n+1})\text{ for all }n\in\mathbb{Z}\}$$
equipped with the induced topology of the product topology of $(E^1)^{\mathbb{Z}}$ (where each copy of $E^1$ is given the discrete topology), and we let $\tss_E:\TSS_E\to\TSS_E$ be the homeomorphism given by $(\tss_E(x))_n=x_{n+1}$ for $x=(x_n)_{i\in\Z}\in\TSS_E$. 

If $E$ and $F$ are finite directed graph $E$ with no sinks or sources, then a \emph{conjugacy} from $\TSS_E$ to $\TSS_F$ is a homeomorphism $\phi:\TSS_E\to\TSS_F$ such that $\tss_F\circ\phi=\phi\circ\tss_E$. The shift spaces $\TSS_E$ and $\TSS_F$ are said to be \emph{conjugate} if there is a conjugacy from $\TSS_E$ to $\TSS_F$. 

For a directed graph $E$, we denote by $\overline k_E$ the map from $(SE)^1$ to $\mathbb{R}$ given by $\overline k_E(e)=1$ for $e\in E^1$, and $\overline k_E(e_{i,v})=0$ for $v\in E^0$ and $i=1,2,\dots$.

\begin{corollary}\label{thm:43}
	Let $E$ and $F$ be directed graphs and let $R$ be a commutative integral domain with identity and $K$ a kind unital $*$-subring of $\C$. The following are equivalent.
	\begin{enumerate}
		\item[(1)] There is an isomorphism $\psi:\G_{SE} \to\G_{SF}$ such that $c_{\overline k_F}\circ\psi=c_{\overline k_E}$.
		\item[(2)] There is a diagonal-preserving $*$-algebra-isomorphism $\phi:L_R(E)\otimes M_\infty(R)\to L_R(F)\otimes M_\infty(R)$ such that $\phi(L_R(E)_n\otimes M_\infty(R))=L_R(F)_n\otimes M_\infty(R)$ for $n\in\mathbb{Z}$.
		\item[(3)] There is a diagonal-preserving ring-isomorphism $\phi:L_R(E)\otimes M_\infty(R)\to L_R(F)\otimes M_\infty(R)$ such that $\phi(L_R(E)_n\otimes M_\infty(R))=L_R(F)_n\otimes M_\infty(R)$ for $n\in\mathbb{Z}$.
		\item[(4)] There is a $*$-ring-isomorphism $\phi:L_K(E)\otimes M_\infty(K)\to L_K(F)\otimes M_\infty(K)$ such that $\phi(L_K(E)_n\otimes M_\infty(K))=L_K(F)_n\otimes M_\infty(K)$ for $n\in\mathbb{Z}$.	
		\item[(5)] There is a diagonal-preserving $*$-isomorphism $\phi:C^*(E)\otimes\mathcal{K}\to C^*(F)\otimes\mathcal{K}$ such that $(\lambda_z^F \otimes\id_\mathcal{K})\circ\phi=\phi\circ(\lambda^E_z\otimes\id_\mathcal{K})$ for $z\in\mathbb{T}$.
	\end{enumerate}
	Suppose that $E$ and $F$ are finite directed graphs with no sinks or sources. Then (1)--(5) are equivalent to the following.
	\begin{enumerate}
		\item[(6)] The two-sided edge shifts $\TSS_E$ and $\TSS_F$ are conjugate.
	\end{enumerate}
\end{corollary}

\begin{proof} Applying Theorem~\ref{thm:2} to the functions $k:E^1\to\mathbb{R}$ and $l:F^1\to\mathbb{R}$ that are constanly equal to 1 shows that (1)--(5) are equivalent. The equivalence of (1) and (6), when $E$ and $F$ are finite directed graphs with no sinks or sources, is proved in \cite[Theorem~5.1]{CR}.
\end{proof}


\begin{corollary}\label{thm:4}
Let $E$ and $F$ be directed graphs, let $R$ be a commutative integral domain with identity and $K$ a kind unital $*$-subring of $\C$. The following are equivalent.
\begin{enumerate}	
	\item[(1)] The groupoids $\G_{SE}$ and $\G_{SF}$ are isomorphic.
	\item[(2)] There is a diagonal-preserving $*$-algebra-isomorphism between $L_R(SE)$ and $L_R(SF)$.
	\item[(3)] There is a diagonal-preserving ring-isomorphism between $L_R(SE)$ and $L_R(SF)$.
	\item[(4)] There is a $*$-ring-isomorphism between $L_K(SE)$ and $L_K(SF)$.
	\item[(5)] There is a diagonal-preserving $*$-isomorphism between $C^*(SE)$ and $C^*(SF)$.
	\item[(6)] There is a diagonal-preserving $*$-algebra-isomorphism between $L_R(E)\otimes M_\infty(R)$ and $L_R(F)\otimes M_\infty(R)$.
	\item[(7)] There is a diagonal-preserving ring-isomorphism between $L_R(E)\otimes M_\infty(R)$ and \linebreak $L_R(F)\otimes M_\infty(R)$.
	\item[(8)] There is a $*$-ring-isomorphism between $L_K(E)\otimes M_\infty(K)$ and $L_K(F)\otimes M_\infty(K)$.
	\item[(9)] There is a diagonal-preserving $*$-isomorphism between $C^*(E)\otimes\mathcal{K}$ and \newline $C^*(F)\otimes\mathcal{K}$.	
\end{enumerate}		
\end{corollary}

\begin{proof}
The equivalence of (1)-(5) follows from Theorem~\ref{thm:1} by letting $k:E^1\to\mathbb{R}$ and $l:F^1\to\mathbb{R}$ be constantly equal to 0, and the equivalence of (1) and (6)--(9) follows from Theorem~\ref{thm:2} by letting $k:E^1\to\mathbb{R}$ and $l:F^1\to\mathbb{R}$ be constantly equal to 0. \end{proof}
 

\begin{remark}
By Remark~\ref{rmk:orbit}, condition (1) is equivalent to the following.
\begin{enumerate}
\item[(10)] There is an orbit equivalence between $SE$ and $SF$ that preserves isolated periodic points.
\end{enumerate}

Moreover, it follows from \cite[Corollary 4.6]{CW} and the discussion right before \cite[Proposition 3.1]{CW} that if either $E$ and $F$ each satisfy condition (L) (i.e., every cycle in $E$ and every cycle in $F$ have an exit), or $E$ and $F$ each have only finitely many vertices and no sinks, then any orbit equivalence between $SE$ and $SF$ automatically preserves isolated periodic points.

It follows from Corollary~\ref{cor:stable} that (6) and (7) are each equivalent to each of the following four conditions.
\begin{enumerate}
	\item[(11)] The groupoids $\G_E\times\mathcal{R}$ and $\G_E\times\mathcal{R}$ are isomorphic.
	\item[(12)] The groupoids $\G_E$ and $\G_F$ are Kakutani equivalent.
	\item[(13)] The groupoids $\G_E$ and $\G_F$ are groupoid equivalent.
	\item[(14)] There are idempotents $p_E\in D_R(E)$ and $p_F\in D_R(F)$ such that $p_E$ is full in $L_R(E)$ and $p_F$ is full in $L_R(F)$, and a ring-isomorphism $\phi:p_EL_R(E)p_E\to p_FL_R(F)p_F$ such that $\phi(p_ED_R(E))=p_FD_R(F)$.
	\end{enumerate}
	
By \cite[Corollary~4.5]{CRS}, condition (12) is equivalent to each of the following two conditions.
\begin{enumerate}
	\item[(15)] There are projections $p_E\in D_R(E)$ and $p_F\in D_R(F)$ such that $p_E$ is full in $L_R(E)$ and $p_F$ is full in $L_R(F)$, and a $*$-algebra-isomorphism $\phi:p_EL_R(E)p_E\to p_FL_R(F)p_F$ such that $\phi(p_ED_R(E))=p_FD_R(F)$.
	\item[(16)] There are projections $p_E\in \mathcal{D}(E)$ and $p_F\in \mathcal{D}(F)$ such that $p_E$ is full in $C^*(E)$ and $p_F$ is full in $C^*(F)$, and a $*$-isomorphism $\phi:p_EC^*(E)p_E\to p_FC^*(F)p_F$ such that $\phi(p_E\mathcal{D}(E))=p_F\mathcal{D}(F)$.		
\end{enumerate}

It follows from \cite[Proposition 5]{Car} that a $*$-ring-isomorphism $\phi:p_EL_K(E)p_E \to\linebreak p_FL_K(F)p_F$ automatically satisfies $\phi(p_ED_K(E))=p_FD_K(F)$. It therefore follows from \cite[Corollary 4.5]{CRS} that (12) is also equivalent to the following.
\begin{enumerate}
	\item[(17)] There are projections $p_E\in D_K(E)$ and $p_F\in D_K(F)$ such that $p_E$ is full in $L_K(E)$ and $p_F$ is full in $L_K(F)$, and a $*$-ring-isomorphism between $p_EL_K(E)p_E$ and $p_FL_K(F)p_F$. 
\end{enumerate}

If $E$ and $F$ only have finitely many vertices, then it follows from \cite[Corollary 4.9]{CRS} that the following condition implies (11).
	\begin{enumerate}
		\item[(18)] The graphs $E$ and $F$ are move equivalent as defined in \cite{Sor}.
	\end{enumerate}
	
	If $E$ and $F$ are finite and have no sinks or sources, then it follows from \cite[Corollary 6.4]{CEOR} that conditions (11) and (18) are equivalent and are also equivalent to the following condition.
	\begin{enumerate}
		\item[(19)] The shifts $\TSS_E$ and $\TSS_F$ are flow equivalent.
	\end{enumerate}

\end{remark}

\section*{Acknowledgements}
James is grateful for the provision of travel support from Aidan Sims and from the University of Wollongong's Institute for Mathematics and its Applications. He is also grateful for the hospitality of Toke and his family and the University of the Faroe Islands during his visit. He also thanks his father Donald for supporting him.


\begin{thebibliography}{99}

\bibitem{Abadie} F. Abadie, \emph{On partial actions and groupoids}, Proc. Amer. Math. Soc. \textbf{132} (2004), 1037--1047.
	

\bibitem{AT} G. Abrams and M. Tomforde, \emph{Isomorphism and Morita equivalence of graph algebras}, Trans. Amer. Math. Soc. \textbf{363} (2011), 3733--3767.

\bibitem{ABHS} P. Ara, J. Bosa, R. Hazrat and A. Sims, \emph{Reconstruction of graded groupoids from graded Steinberg algebras}, to appear in Forum Math., \texttt{doi:10.1515/forum-2016-0072}, 15 pages.

\bibitem{AP} P. Ara and F. Perera, \emph{Multipliers of von Neumann regular rings}, Comm. Algebra \textbf{28} (2000), 3359--3385.


\bibitem{AER} S. Arklint, S. Eilers, and E. Ruiz, \emph{A dynamical characterization of diagonal preserving $*$-isomorphisms of graph $C^*$-algebras}, to appear in Ergodic Theory Dynam. Systems, \texttt{doi:10.1017/etds.2016.141}, 21 pages.

\bibitem{BCH} J.H. Brown, L. Clark, and A. an Huef, \emph{Diagonal-preserving ring $*$-isomorphisms of Leavitt path algebras}, J. Pure Appl. Algebra (10) \textbf{221} (2017), 2458--2481.

\bibitem{BCW} N. Brownlowe, T.M. Carlsen, and M.F. Whittaker, \emph{Graph algebras and orbit equivalence},  Ergodic Theory Dynam. Systems (2) \textbf{37} (2017), 389--417


\bibitem{Car} T.M. Carlsen, \emph{$*$-isomorphism of Leavitt path algebras over $\mathbb{Z}$}, \texttt{arXiv:1601.00777v2}, 8 pages.

\bibitem{CEOR} T.M. Carlsen, S. Eilers, E. Ortega and G. Restroff, \emph{Flow equivalence and orbit equivalence for shifts of finite type and isomorphism of their groupoids}, \texttt{arXiv:1610.09945v3}, 23 pages.

\bibitem{COP} T.M. Carlsen, E. Ortega and E. Pardo, \emph{$C^*$-algebras associated to Boolean dynamical systems}, J. Math. Anal. Appl. (1) \textbf{450} (2017), 727--768.

\bibitem{CR} T.M. Carlsen and J. Rout, \emph{Diagonal-preserving gauge-invariant isomorphisms of graph $C^*$-algebras}, \texttt{arXiv:1610.00692v1}, 12 pages.

\bibitem{CRS} T.M. Carlsen, E. Ruiz, and A. Sims, \emph{Equivalence and stable isomorphism of groupoids, and diagonal-preserving stable isomorphisms of graph $C^*$-algebras and Leavitt path algebras}, Proc. Amer. Math. Soc. \textbf{145} (2017), 1581--1592.

\bibitem{CW} T.M Carlsen and M.L. Winger, \emph{Orbit equivalence of graphs and isomorphism of graph groupoids}, to appear in Math. Scand., \texttt{arXiv:1610.09942v2}, 9 pages.

\bibitem{CE} L.O. Clark and C. Edie-Michell, \emph{Uniqueness theorems for Steinberg algebras}, Algebr. Represent. Theory \textbf{18} (2015), 907--916.

\bibitem{CEHS} L.O. Clark, C. Edie-Michell, A. an Huef, and A. Sims, \emph{Ideals of Steinberg algebras of strongly effective groupoids, with applications to Leavitt path algebras}, \texttt{arXiv:1601.07238v1}, 25 pages.

\bibitem{CEP} L.O. Clark, R. Exel, and E. Pardo, \emph{A Generalised uniqueness theorem and the graded ideal structure of Steinberg algebras}, to appear in Forum Math., \texttt{arXiv:1609.02873v1}, 23 pages.

\bibitem{CFST} L.O. Clark, C. Farthing, A. Sims, and M. Tomforde, \emph{A groupoid generalisation of Leavitt path algebras}, Semigroup Forum \textbf{89} (2014), 501--517.

\bibitem{CMMS} L.O. Clark, D. Martin Barquero, C. Martin Gonzalez, and M. Siles Molina, \emph{Using Steinberg algebras to study decomposability of Leavitt path algebras}, to appear in Forum Math., \texttt{doi:10.1515/forum-2016-0062}, 17 pages.

\bibitem{CMMS2} L.O. Clark, D. Martin Barquero, C. Martin Gonzalez, and M. Siles Molina, \emph{Using the Steinberg algebra model to determine the center of any Leavitt path algebra}, \texttt{arXiv:1604.01079v1}, 14 pages.

\bibitem{CP} L.O. Clark and Y.E.P. Pangalela, \emph{Kumjian-Pask algebras of finitely-aligned higher-rank graphs}, J. Algebra \textbf{482} (2017), 364--397.

\bibitem{CS} L.O. Clark and A. Sims, \emph{Equivalent groupoids have Morita equivalent Steinberg algebras}, J. Pure Appl. Algebra \textbf{219} (2015), 2062--2075.





\bibitem{ERRS} S. Eilers, G. Restorff, E. Ruiz and A.P.W. Sørensen, \emph{The complete classification of unital graph $C^*$-algebras: Geometric and strong}, \texttt{arXiv:1611.07120}, 73 pages.

\bibitem{Exel} R. Exel, \emph{Inverse semigroups and combinatorial $C^\ast$-algebras}, Bull. Braz. Math. Soc. (N.S.) \textbf{39} (2008), 191--313.

\bibitem{EP} R. Exel and E. Pardo, \emph{The tight groupoid of an inverse semigroup}, Semigroup Forum \textbf{92} (2016), 274--303.



\bibitem{Higman} G. Higman, \emph{The units of group-rings}, Proc. London Math. Soc. \textbf{46} (1940), 231-–248. 






\bibitem{MCon} K. Matsumoto, \emph{Continuous orbit equivalence, flow equivalence of Markov shifts and circle actions on Cuntz–Krieger algebras}, Math. Z. \textbf{285} (2017) 121--141.





\bibitem{Matui} H. Matui, \emph{ Homology and topological full groups of \'etale groupoids on totally disconnected spaces}, Proc. London Math. Soc. (3) \textbf{104} (2012), 27--56.
	
\bibitem{MRW} P.S. Muhly, J.N. Renault and D.P. Williams, {\em Equivalence and isomorphism for groupoid $C^\ast$-algebras}, J. Operator Theory \textbf{17} (1987), 3--22.






\bibitem{Sor} A.P.W Sørensen, {\em Geometric classification of simple graph algebras}, Ergodic Theory Dynam. Systems {\bf 33} (2013), 1199--1220.

\bibitem{S} B. Steinberg, \emph{A groupoid approach to discrete inverse semigroup algebras}, Adv. Math. \textbf{223} (2010), 689--727.

\bibitem{S1} B. Steinberg, \emph{Modules over \'etale groupoid algebras as sheaves}, J. Aust. Math. Soc. \textbf{97} (2014), 418--429.

\bibitem{S2} B. Steinberg, \emph{Simplicity, primitivity and semiprimitivity of \'etale groupoid algebras with applications to inverse semigroup algebras}, J. Pure Appl. Algebra \textbf{220} (2016), 1035--1054.

\bibitem{T2}  M. Tomforde, \emph{Stability of $C^\ast$-algebras associated to graphs} Proc. Amer. Math. Soc. \textbf{132} (2004), 1787--1795.


\bibitem{Tomforde2} M. Tomforde, \emph{Leavitt path algebras with coefficients in a commutative ring }, J. Pure Appl. Algebra {\bf 215} (2011), 471--484.

\bibitem{Web} S. Webster, {\em The path space of a directed graph}, Proc. Amer. Math. Soc. {\bf 142} (2014), 213--225.

\bibitem{Y} T. Yeend, \emph{Groupoid models for the $C^*$-algebras of topological higher-rank graphs}, J. Operator Theory \textbf{57} (2007) 95-–120.

\end{thebibliography}
\end{document}